\documentclass[11pt, reqno, english]{amsart}
\usepackage{comment}
\usepackage{accents}
\usepackage{style}
\usepackage[normalem]{ulem}
\usepackage{wasysym}
\usepackage{amssymb}
\usepackage{graphicx}
\usepackage{calc}
\usepackage{tikz}
\usepackage{charter}
\usepackage{euler}
\usepackage[T1]{fontenc}
\usepackage[bb=stixtwo]{mathalpha}
\usetikzlibrary{decorations.pathreplacing,calligraphy}

\usepackage{thmtools}
\usepackage{thm-restate}

\definecolor{myblue}{RGB}{35,125,225}

\begin{document}

\author{Henry Adams}
\email{henry.adams@ufl.edu}

\author{Armando Albornoz}
\email{albornoz.a1@ufl.edu}

\author{Glenn Bruda}
\email{glenn.bruda@ufl.edu}

\author{Jianda Du}
\email{jianda.du@ufl.edu}

\author{Lodewyk Jansen van Rensburg}
\email{jansenva@ufl.edu}

\author{Alejandro Leon}
\email{leon.alejandro@ufl.edu}

\author{Saketh Narayanan}
\email{saketh.narayanan@ufl.edu}

\author{Connor Panish}
\email{wilsonc1@ufl.edu}

\author{Chris Rugenstein}
\email{crugenstein@ufl.edu}

\author{Martin Wall}
\email{martinwall@ufl.edu}

\title{
Gromov--Hausdorff distances between quotient metric spaces
}

\begin{abstract}
The Hausdorff distance measures how far apart two sets are in a common metric space.
By contrast, the Gromov--Hausdorff distance provides a notion of distance between two abstract metric spaces.
How do these distances behave for quotients of spaces under group actions?
Suppose a group $G$ acts by isometries on two metric spaces $X$ and $Y$.
In this article, we study how the Hausdorff and Gromov--Hausdorff distances between $X$ and $Y$ and their quotient spaces $X/G$ and $Y/G$ are related.
For the Hausdorff distance, we show that if $X$ and $Y$ are $G$-invariant subsets of a common metric space, then we have $d_\h(X,Y)=d_\h(X/G,Y/G)$.
However, the Gromov--Hausdorff distance does not preserve this relationship: we show how to make the ratio $\frac{d_\gh(X/G,Y/G)}{d_\gh(X,Y)}$ both arbitrarily large and arbitrarily small, even if $X$ is an arbitrarily dense $G$-invariant subset of $Y$.
\end{abstract}

\maketitle


\section{Introduction}
Metric spaces, generalizing the intuitive notion of Euclidean distance, allow us to quantify the separation or the distance between two elements in that space.
But a more general question arises: how can we extend this notion to quantify the dissimilarity or the distance between two \emph{subsets} of a metric space?
One answer is the Hausdorff distance, introduced in Hausdorff's 1914 book \emph{Grundzüge der Mengenlehre}~\cite{hausdorff1914grundzuge}.
The Hausdorff distance in fact generates a metric space structure on the set of all nonempty compact subsets of a metric space; thus the Hausdorff distance provides a natural generalization of a metric space's distance function.

However, the Hausdorff distance is only defined for two subsets living in a common metric space.
To address the challenge of measuring the separation between \emph{arbitrary} metric spaces, we turn to the more complicated Gromov--Hausdorff distance; see~\cite{edwards1975structure,gromov1981groups, gromov1981structures,tuzhilin2016invented,berger2000encounter}.
By minimizing the Hausdorff distance across all distance-preserving mappings into a common space, the Gromov--Hausdorff distance provides a way to measure the separation between two metric spaces.
Moreover, the Gromov--Hausdorff distance turns the set of all isometry classes of compact metric spaces into a metric space.
The Gromov--Hausdorff distance is a natural distance in the area of object matching under invariances~\cite{ms04, ms05}, and the motivation for many other distances in the applied literature.

In this paper, we consider metric spaces equipped with group actions, and we explore how the Hausdorff and Gromov--Hausdorff distances behave in this context.
Indeed, let $M$ be a metric space, and let $G$ be a group that acts on $M$ by isometries.
If the action of $G$ on $M$ is \emph{proper}, which is always the case if the group $G$ is finite, then we can equip the quotient space $M/G$ with an induced metric.
If $X$ and $Y$ are $G$-invariant subsets of $M$, then how does the Hausdorff distance between $X$ and $Y$ compare to the Hausdorff distance between $X/G$ and $Y/G$?
How does the Gromov--Hausdorff distance between $X$ and $Y$ compare to the Gromov--Hausdorff distance between $X/G$ and $Y/G$?
We provide results showing how these distances are related.

Our main results are as follows.
For the Hausdorff distance, we show in Theorem~\ref{thm:Hausdorff} that if $X$ and $Y$ are $G$-invariant subsets of a common metric space, then we have $d_\h(X,Y)=d_\h(X/G,Y/G)$.
However, the Gromov--Hausdorff distance does not preserve this relationship: we show in Theorem~\ref{thm:gh-ratio} that the ratio $\frac{d_\gh(X/G,M/G)}{d_\gh(X,M)}$ can be made both arbitrarily large and arbitrarily small, even when $X$ is a subset of a finite metric space $M$.
Can one do better if $X$ is an arbitrarily dense subset of $M$, i.e.\ if we make the Hausdorff distance $d_\h(X,M)$ arbitrarily small?
We show the answer is no.
Indeed, in Theorem~\ref{thm:gh-ratio-dense} we show that for any positive integer $n$, there is a compact metric space $M$ with a proper isometric $G$ action such that for any $\delta>0$, there is a $G$-invariant subset $X\subseteq M$ with $d_\h(X,Z)<\delta$ satisfying $\frac{d_\gh(X/G,M/G)}{d_\gh(X,M)}\ge n$.

In addition, we leave behind the setting of group actions in order to address a question from~\cite{HvsGH}.
Theorem~1(a) of~\cite{HvsGH} states that if $X$ is a sufficiently dense sample from a compact manifold $M$, then $d_\gh(X,M)$ is bounded from below by a reasonable constant times $d_\h(X,M)$.
In Theorem~\ref{thm:manifold-assumtion-necessary} we show that for any $\varepsilon>0$, there is a compact metric space $Z$ such that for any $\delta>0$, there is a subset $X\subseteq Z$ with $d_\h(X,Z)<\delta$ with $d_\gh(X,Z) < \varepsilon \cdot d_\h(X,Z)$.
Our result therefore shows that the manifold assumption in~\cite[Theorem~1(a)]{HvsGH} is necessary.

We present our results in increasing levels of complexity, sometimes starting with a weaker statement with a simpler proof before progressing to a stronger statement with a more complicated proof.

As partial motivation for this paper, suppose that one is given a dataset with symmetries or approximate symmetries.
For example, the dataset may be a laser scan of the entrance to a museum.
Certain portions of the museum entrance may exhibit symmetries, such as the four columns to the left of the entrance being an approximate reflection of the four columns to the right of the entrance.
Similarly, the octagonal windows above the entrance may exhibit eight-fold rotational symmetry.
Portions of the laser scan may be incomplete, for example obscured by a tree branch, and other portions of the laser scan may be less detailed than in other areas.
Can one identify symmetries or approximate symmetries in this dataset, and then use that structure in order to improve the detail or to inpaint missing regions?
For example, one could use portions of an unobscured column in order to inpaint an obscured column.
Similarly, one could rotate one pane of an octagonal window around to add more detail to the scans of other panes.
One might identify that one portion of the dataset is similar to another portion, such as one column being similar to another, using notions of distance like the Hausdorff or Gromov--Hausdorff distance.
The goal of our paper is not to tackle this application directly (see~\cite{mgp_approx_symm_sig_06,graczyk2022model,huang2024approximately,korman2015probably} for related work on this topic).
Instead, our goal is to first develop the mathematical theory for Hausdorff and Gromov--Hausdorff distances between spaces and subsets thereof that are equipped with group actions.
We hope that expansions upon our work may later find use in applications similar to the ones described above.

We begin in Section~\ref{sec:related} by surveying related work.
In Section~\ref{sec:background} we introduce the necessary background definitions and notation, including concepts such as metric spaces, the Hausdorff distance, the Gromov--Hausdorff distance, and group actions.
In Section~\ref{sec:results-H} we show how the Hausdorff distance is well-behaved with respect to group quotients, whereas in Section~\ref{sec:results-GH} we show that the Gromov--Hausdorff distance is not.
Furthermore, in Section~\ref{sec:results-density} we show that adding the assumption that one metric space is an arbitrarily dense subset of the other metric space does not save these Gromov--Hausdorff distance results.
In Section~\ref{sec:addressing-a-question}, we answer an open question from~\cite{HvsGH}.
In Section~\ref{sec:conclusion} we conclude by organizing a list of open questions, as we hope that one of the contributions of this article is to provide a platform for future research on these topics.

\section{Related work}
\label{sec:related}

We refer the reader to~\cite{edwards1975structure,gromov1981groups, gromov1981structures,tuzhilin2016invented,berger2000encounter} for the history of the Gromov--Hausdorff distance, beginning with a 1975 paper by Edwards.
The first definition of the Gromov--Hausdorff distance between two metric spaces $X$ and $Y$ that one typically sees takes an infimum over all isometric embeddings of $X$ and $Y$ into some other metric space $Z$ (see Definition~\ref{def:GH}).
However, when working with Gromov--Hausdorff distances, one often instead uses an equivalent definition that infimizes over all correspondences between $X$ and $Y$ (see Theorem~\ref{thm:distortion_GH}).
It is known that if the input metric spaces are compact, then these infima are in fact realized; see for example~\cite{ivanov2016realizations}.

In applications to Riemannian geometry, one is often interested in the Gromov--Hausdorff limit of a sequence of Riemannian manifolds.
The Gromov--Hausdorff distance allows these spaces to be considered intrinsically.
Without additional assumptions, the Gromov--Hausdorff limit of a sequence of manifolds may not be well-behaved.
But properties of the limit can be proven in the presence of additional assumptions, say on the curvature or the volume, of the manifolds in the sequence; see for example~\cite{sormani2011intrinsic}.

It is difficult to compute, or even to approximate, the Gromov--Hausdorff distance between two metric spaces~\cite{agarwal2018computing,schmiedl2015shape,schmiedl2017computational}.
For this reason, researchers have considered relaxed variants thereof.
In~\cite{memoli2008euclidean}, M\'emoli considers the Gromov--Hausdorff distance $d_\gh$ and the Hausdorff distance $d_\h$ under the action of Euclidean isometries, motivating and defining the Euclidean--Hausdorff distance $d_{\mathrm{EH}}$ by modifying the expression for the Gromov--Hausdorff distance.
M\'emoli proves that $d_\gh(X,Y)\leq d_{\mathrm{EH}}(X,Y)\leq c_n\sqrt{\max\{\diam X,\diam Y\}\cdot d_\gh(X,Y)}$ for any compact $X,Y\subseteq \R^n$, where $c_n$ is constant depending on the dimension $n$.

In the paper~\cite{HvsGH}, the authors provide a lower bound $d_\gh(X,M)\geq 1/2\cdot d_\h(X,M)$ on the Gromov--Hausdorff distance between a closed Riemannian manifold $M$ and a sufficiently dense subset $X\subseteq M$.
For the special case $M=S^1$ with condition $d_\gh(X,S^1)<\pi/6$, they obtained $d_\gh(X,S^1) =  d_\h(X,S^1)$, proving that the Hausdorff and Gromov--Hausdorff distances coincide in this setting.
They also provide lower bounds on the Gromov--Hausdorff distance $d_\gh(X, Y)$ between two subsets $X$ and $Y$ of $M$.

See~\cite{dong2021gromov,khan2018gh} for papers relating Gromov--Hausdorff stability to group actions.

\section{Background and notation}
\label{sec:background}

We provide some background information and notation on metric spaces, the Hausdorff and Gromov--Hausdorff distances, correspondences, group actions, and quotient metric spaces.
We refer the reader to~\cite{bridson2011metric,BuragoBuragoIvanov} for more information.

\subsection*{Metric spaces}
Metric spaces encode distances between points in the space.

\begin{definition}[Metric space]
A \emph{metric space} is a set M equipped with a function $d\colon M\times M \rightarrow \R$ such that for all $x,y,z\in{M}$ the following are satisfied:
\begin{itemize}
    \item $d(x,x) = 0$ for all $x$, and $d(x,y) > 0$ for $x \neq y$.
    \item $d(x,y) = d(y, x)$.
    \item $d(x,z) \leq d(x,y) + d(y,z)$.
\end{itemize}
\end{definition}

\subsection*{The Hausdorff distance}

We define the Hausdorff distance using epsilon-thickenings, although there are other equivalent definitions.

\begin{definition}[Epsilon-thickenings]
Let $Z$ be a metric space and $X \subseteq Z$.
Then the \emph{$\varepsilon$-thickening of $X$}, denoted $X^\varepsilon$, is the set $X^\varepsilon \coloneqq \{z \in Z~\mid~d(z , x) < \varepsilon \text{ for some } x \in X \}$.
\end{definition}

\begin{definition}[Hausdorff distance]
If $X$ and $Y$ are two subsets of the same metric space $(Z,d)$, then the \emph{Hausdorff distance} between them is
\[d_\h(X,Y)\coloneqq \inf\{\varepsilon > 0 ~|~ X \subseteq Y^{\varepsilon} \text{ and } Y \subseteq X^{\varepsilon}\}\]
\end{definition}

An equivalent definition of the Hausdorff distance is the following:
\[d_\h(X,Y)=\max\left\{\sup_{x\in X}\inf_{y\in Y}d(x,y),~\sup_{y\in Y}\inf_{x\in X}d(x,y)\right\}.\]
The Hausdorff distance gives a metric structure on the compact subsets of a metric space.

\subsection*{The Gromov--Hausdorff distance}

Using the Hausdorff distance and isometric embeddings, we can now define the Gromov--Hausdorff distance.

\begin{definition}
\label{def:isometric-embedding}
A mapping between two metric spaces is an \textit{isometry} if it preserves distance.
In symbols, $f\colon X \rightarrow Y$ is an \emph{isometry} if $X$ and $Y$ are metric spaces and if $d_X(a,b) = d_Y(f(a), f(b))$ for all $a,b \in X$.
\end{definition}

Although some authors require an isometry to be a bijection, we do not require this.
However, notice that isometries are always injective.

\begin{definition}[Gromov--Hausdorff distance]
\label{def:GH}
If $X$ and $Y$ are two metric spaces, then the \emph{Gromov--Hausdorff} distance between them is
\[d_\gh(X,Y) \coloneqq \underset{\underset{f\colon X \rightarrow Z, g\colon Y \rightarrow Z}{\text{isometries}}}{\text{infimum}} \Big\{d^{Z}_\h(f(X),g(Y))\Big\}.\]
\end{definition}

So, over all possible isometries $f\colon X\to Z$ and $g\colon Y \to Z$ into a common metric space $Z$, the Gromov--Hausdorff distance $d_\gh$ is the infimal possible Hausdorff distance between the two embeddings $f(X)$ and $g(Y)$.
This infimum can be recast as an infimum over all metrics on $X\coprod Y$ that extend the metrics on $X$ and $Y$ --- hence as an infimum over a set instead of over a proper class~\cite[Remark~7.3.12]{BuragoBuragoIvanov}.

It is rarely easy to compute the exact Gromov--Hausdorff distance between two metric spaces.
It is also often difficult to even estimate the Gromov--Hausdorff distance; see for example~\cite{schmiedl2015shape,schmiedl2017computational}.
This is by nature of the number of isometries and metric spaces which must be considered.
In addition, the Gromov--Hausdorff distance can behave unintuitively at first, as we describe in the example below.

\begin{example}
Let $X$ be the set of 3 points that define the vertices of an equilateral triangle with side length $1$.
Let $Y$ be a single point; see Figure~\ref{fig:triangle-example}.
Then $X$ and $Y$ are realized in Euclidean space; however, their Gromov--Hausdorff distance is not.

\begin{figure}
\begin{center}
\scalebox{0.7}{
\begin{tikzpicture}
    \draw[densely dashed,gray] (0,0)--(8,0) node[]{};
    \draw[densely dashed,gray] (8,0)--(4,6.928) node[]{};
    \draw[densely dashed,gray] (4,6.928)--(0,0) node[]{};
    \draw[densely dashed,gray] (4,6.928)--(4,2.308) node[]{};
    \draw[densely dashed,gray] (0,0)--(4,2.308) node[]{};
    \draw[densely dashed,gray] (8,0)--(4,2.308) node[]{};

    \draw[black] (0,0)--(0,7.2) node[]{\tiny$\blacktriangle$};
    \draw[black] (0,0)--(8.3,0) node[]{\tiny$\mathbin{\rotatebox[origin=c]{-90}{$\blacktriangle$}}$};
        
    \draw[red] (8,0)--(8,0) node[draw, shape = circle, fill = red, minimum size = 0.2cm, inner sep=0pt]{};
    \draw[red] (0,0)--(0,0) node[draw, shape = circle, fill = red, minimum size = 0.2cm, inner sep=0pt]{};
    \draw[red] (4,6.928)--(4,6.928) node[draw, shape = circle, fill = red, minimum size = 0.2cm, inner sep=0pt]{};
    \draw[myblue] (4,2.308)--(4,2.308) node[draw, shape = circle, fill = blue, minimum size = 0.2cm, inner sep=0pt]{};

    \node[left] at (1.8,3.566) {\Huge$1$};
    \node[right] at (3.95,4.316) {\huge$\frac{1}{\sqrt{3}}$};
\end{tikzpicture}
}
\scalebox{0.7}{
\begin{tikzpicture}
    \draw[black] (0,0)--(3,3) node[]{};
    \draw[black] (6,0)--(3,3) node[]{};
    \draw[black] (3,3)--(3,7.24) node[]{};

    \draw[red] (0,0)--(0,0) node[draw, shape = circle, fill = red, minimum size = 0.2cm, inner sep=0pt]{};
    \draw[red] (3,7.24)--(3,7.24) node[draw, shape = circle, fill = red, minimum size = 0.2cm, inner sep=0pt]{};
    \draw[red] (6,0)--(6,0) node[draw, shape = circle, fill = red, minimum size = 0.2cm, inner sep=0pt]{};
    \draw[myblue] (3,3)--(3,3) node[draw, shape = circle, fill = blue, minimum size = 0.2cm, inner sep=0pt]{};

    \node[left] at (1.4,2.2) {\Huge$\frac{1}{2}$};
    \node[right] at (4.5,2.2) {\Huge$\frac{1}{2}$};
    \node[right] at (3.2,5.12) {\Huge$\frac{1}{2}$};
\end{tikzpicture}
}

\caption{The sets $X$ (in red) and $Y$ (in blue) are shown embedded in the plane on the left.
On the right, they are embedded into a graph $G$.}

\label{fig:triangle-example}
\end{center}
\end{figure}
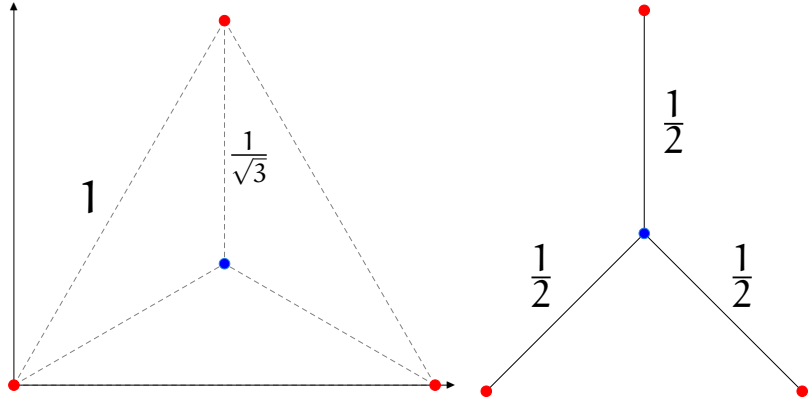

In the plane, the smallest achievable Hausdorff distance is $\frac{1}{\sqrt{3}}$.
This Hausdorff distance is attained by centering the blue point within the triangle defined by $X$, as shown in Figure~\ref{fig:triangle-example}.
But now consider embedding $X$ and $Y$ into the graph $G$, which is a the star graph with one internal (or central) node, 3 leaves, and three edges of length $\frac{1}{2}$.
We map the single blue point in $Y$ to the internal node and each point in $X$ to a different leaf.
We can check that this is, in fact, an isometry.
Observe that the Hausdorff distance between these embedded images of $X$ and $Y$ is $\frac{1}{2}$.
Therefore $d_\gh(X,Y)\le\frac{1}{2}$, and this cannot be achieved within Euclidean space.
\end{example}

\subsection*{Correspondences}

We now use correspondences to give an equivalent definition of the Gromov--Hausdorff distance, which is arguably more useful for computations.

\begin{definition}[Correspondence]

A \emph{correspondence} \textbf{C} between any two nonempty sets $X$ and $Y$ is defined to be a subset of $X \times Y$ with the following properties:

\begin{enumerate}[(i)]
\item For any $x \in X$, there exists a $y \in Y$ such that $(x,y) \in \textbf{C}$, and
\item For any $y \in Y$, there exists an $x \in X$ such that $(x,y) \in \textbf{C}.$
\end{enumerate}
\end{definition}

In other words, a correspondence is a relation that assigns to each element of $X$ and $Y$ at least one corresponding element of the opposite set.
It can be thought of as a ``double surjection''.
Note that from a surjective function $f \colon X \to Y$, one can obtain a correspondence $\textbf{C}$ defined as $\textbf{C}=\{(x,f(x))~|~x\in X\}$.

\begin{definition}[Additive Distortion]
If $\textbf{C}$ is a correspondence between two metric spaces $X$ and Y, then its \emph{distortion}, denoted $\dis(\mathbf{C}$), is defined to be the following:

\[\dis(\mathbf{C})\coloneqq\sup_{(x_1,y_1), (x_2,y_2) \in \textbf{C}} |d_X(x_1,x_2) - d_Y(y_1,y_2)|.\]
\label{def:AdditiveDistortion}
\end{definition}

This definition measures how much the correspondence $\textbf{C}$ \textit{distorts} or changes distances amongst corresponding points in $X$ and in $Y$.
A high distortion signifies that the correspondence relates points in $X$ at distance $\delta_X$ apart to points in $Y$ at distance $\delta_Y$ apart, where $\delta_X$ and $\delta_Y$ differ by a substantial margin.
Also notice that an isometry that is a surjection (and hence a bijection) determines a correspondence that has zero distortion since distance is preserved, meaning $d_X(x,x')=d_Y(f(x),f(x'))$.

\subsection*{An equivalent definition of the Gromov--Hausdorff distance}

We now give an equivalent definition of the Gromov--Hausdorff distance that is often convenient to work with:

\begin{theorem}[{{\cite{bridson2011metric,BuragoBuragoIvanov,kalton1999distances}}}]
\label{thm:distortion_GH}
Given metric spaces $X$ and $Y$, the Gromov--Hausdorff distance between them can be equivalently defined as
\begin{align*}
    d_\gh(X, Y) = \tfrac{1}{2}\cdot \inf_{R}(\dis(R)), 
\end{align*}
where the infimum is taken over all correspondences $R$ between $X$ and $Y$.
\end{theorem}

The distortion measures how well a correspondence between $X$ and $Y$ preserves distances.
If the Gromov--Hausdorff distance between two compact metric spaces $X$ and $Y$ is zero, with the infimum realized, then $X$ and $Y$ are isometric.
So taking the infimum over all possible correspondences measures how close $X$ and $Y$ are to being isometric.
If $X$ and $Y$ are compact, then this infimum is in fact realized~\cite{ivanov2016realizations}.

\subsection*{Initial bounds on the Gromov--Hausdorff distance}

For $X$ a metric space, we define the \emph{diameter} of $X$ to be $\diam(X)\coloneqq \sup_{x,x'\in X}d(x,x')$.

Let $X$ and $Y$ be metric spaces.
The diameters of $X$ and $Y$ place initial upper and lower bounds on the Gromov--Hausdorff distance between them.
Indeed, Example~6.30 of~\cite{tuzhilin2020lectures} explains how we have the upper bound $d_\gh(X,Y)\le\frac{1}{2}\max\{\diam(X),\diam(Y)\}$.
Also, if either $X$ or $Y$ have finite diameter, then Example~6.29 of~\cite{tuzhilin2020lectures} gives the lower bound $d_\gh(X,Y)\ge \frac{1}{2}|\diam(X)-\diam(Y)|$.
In particular, if $X$ is a singleton, then $d_\gh(X,Y)=\frac{1}{2}\diam(Y)$.
However, if $X$ and $Y$ have the same diameter, then it is more difficult to place nontrivial bounds on $d_\gh(X,Y)$.

\subsection*{Group actions}
For more information on group actions, see for example~\cite{dummit2004abstract}.

A \emph{group} is a set $G$ along with a binary operation $\cdot$ such that:
\begin{itemize}
\item the operation $\cdot$ is associative, meaning $(a\cdot b)\cdot c=a \cdot (b\cdot c) $ for all $a,b,c\in G$,
\item there exists an identity element $e\in G$ such that $a\cdot e=e\cdot a=a $ for all $a\in G$,
\item for each $a\in G$, there exists an inverse $a^{-1}\in G$ such that $a\cdot a^{-1}=a^{-1}\cdot a=e$.
\end{itemize}

The sets $\Z, \R, \mathbb{Q}$ form groups under addition, 
and the sets $\mathbb{Q}-\{0\}$, $\R-\{0\}$, and $\R^{+}$ form groups under multiplication.
For $n$ a positive integer $n$, the set $\Z/n=\{0,1,\ldots,n-1\}$ forms a group under the addition of integers modulo $n$.

\begin{definition} [Group Action]
Let $G$ be a group, let $S$ be a set, and let $G \times S \rightarrow S$ be a map, denoted $(g,x)\mapsto g\cdot x$.
We say $G$ \emph{acts} on the set $S$ when, for all $x \in S$:
\begin{itemize}
\item $e \cdot x = x$, where $e$ is the identity element of $G$,
\item $g \cdot (h \cdot x) = (gh) \cdot x$ for all $g,h\in G$.
\end{itemize}
\end{definition}

Here, $G \times S \rightarrow S$ is called a \emph{(left) group action} on $S$.
Note that every $g\in G$ acts as a permutation on $S$.
Some example group actions are as follows.
\begin{itemize}
\item The $\Z/n$ action on the vertices of a regular $n$-gon in the plane by rotations,
\item the $\Z/2$ action on Euclidean $n$-space by reflection about the origin, and
\item the action of a group $G$ on itself, either by left multiplication ($g\cdot x = gx$ for all $g,x\in G$) or by conjugation ($g\cdot x=gxg^{-1}$ for all $g,x\in G$).
\end{itemize}

\begin{definition}[$G$-invariance]
Let $G$ be a group acting on a set $X$, and let $Y$ be a subset of $X$.
Let $G \cdot Y$ denote the set $\{g \cdot y: g \in G, y \in Y\}$.
We say $Y$ is \emph{$G$-invariant} if $Y = G \cdot Y$.
\end{definition}

\subsection*{Quotient metric spaces}

We next consider group actions on metric spaces; we refer the reader to~\cite{bridson2011metric,BuragoBuragoIvanov} for more details.

\begin{definition}[Isometric action]
A group $G$ acting on a metric space $(X,d)$ is an \emph{isometric action} if for every $g\in G$, the map $X\to X$ defined by $x\mapsto g\cdot x$ is an isometry.
\end{definition}

We also say that $G$ acts \emph{isometrically} or \emph{by isometries}.

\begin{definition}[Proper action]
The action of $G$ on $X$ is \emph{proper} if, for each $x \in  X$, there exists some $r > 0$ such that the set $\{g \in G~|~(g \cdot B(x, r)) \cap B(x, r) \neq \emptyset\}$ is finite.
\end{definition}

In particular, an action by a finite group is necessarily proper.

\begin{definition}[Quotient metric]
Let $G$ be a group acting properly by isometries on a metric space $(M,d_M)$.
We obtain a metric space $(M/G, d_{M/G})$ by equipping the quotient $M/G$ with the \emph{quotient metric}, which is defined as
\[
    d_{M/G}([x], [x']) = \inf_{g \in G} d_M(x, g \cdot x').
\]
\end{definition}

Here $[x]\in M/G$ is notation for the orbit generated by $x\in M$.
For a proof that this indeed gives a metric on $M/G$, see for example~\cite[I.8.5]{bridson2011metric}.
For the remainder of this article, we focus on proper isometric actions, so that $M/G$ is a quotient metric space.

\section{The Hausdorff distance under quotients}
\label{sec:results-H}

In this paper, we explore questions involving Hausdorff and Gromov--Hausdorff distances when $M$ is a metric space, when $G$ is a group acting properly on $M$ by isometries, and when $X$ and $Y$ are $G$-invariant subsets of $M$.
In particular, we explore the Hausdorff and Gromov--Hausdorff distances between the spaces $M$, $X$, $Y$, $M/G$, $X/G$, and $Y/G$, presenting our results with increasing levels of complexity.

In this section, we show that the Hausdorff distance between $X$ and $Y$ is equal to the Hausdorff distance between the quotient spaces, namely $d_\h(X,Y)=d_\h(X/G,Y/G)$.
In other words, the Hausdorff distance is unaffected by the quotient by the group action.
Though this result may be expected since $G$ acts by isometries, we will see in Section~\ref{sec:results-GH} that changing from the Hausdorff distance to the Gromov--Hausdorff distance now allows the ratio $\frac{d_\gh(X/G,Y/G)}{d_\gh(X,Y)}$ to be made both arbitrarily large and arbitrarily small.

Before proving $d_\h(X,Y)=d_\h(X/G,Y/G)$, we first introduce a lemma establishing the relationship between $G$-invariance of subsets and their thickenings; see Figure~\ref{fig:G-inv-thickening}.

\begin{lemma}
\label{lem:G-balls}
Let $X$ be a $G$-invariant subset of a metric space $M$ on which $G$ acts properly and isometrically.
Then for any $\varepsilon > 0$, the $\varepsilon$-thickening $X^\varepsilon$ is also $G$-invariant.
\end{lemma}

\begin{proof}
We show that $G \cdot X^\varepsilon \coloneqq \{g \cdot x' ~:~ g \in G, x' \in X^\varepsilon\}$ is equal to $X^\varepsilon$.
Take $g \cdot x' \in G \cdot X^\varepsilon$.
By definition of $X^\varepsilon$, it must be that $d(x, x') < \varepsilon$ for some $x \in X$.
The group $G$ acts by isometries on $X$, so $d(x, x') = d(g \cdot x, g \cdot x') < \varepsilon$.
As $X$ is $G$-invariant, $g \cdot x \in X$, so $g \cdot x' \in X^\varepsilon$.
Thus we have $G \cdot X^\varepsilon \subseteq X^\varepsilon$.
Also, it is easy to see that $X^\varepsilon \subseteq G \cdot X^\varepsilon$, since for $\Bar{x} \in X^\varepsilon$, we have that $\Bar{x}=e \cdot \Bar{x} \in G \cdot X^\varepsilon$.
\end{proof}

\begin{figure}
\centering
\centering
\includegraphics[width=2.6in]{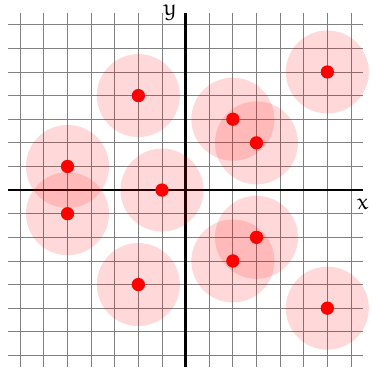}
\caption{Let $X$ be the set of red points, let $X^\varepsilon$ be the light red neighborhoods around them, and let $G$ be the group which reflects points across the x-axis.
Then $X$ is $G$-invariant and so $X^\varepsilon$ must also be $G$-invariant.}
\label{fig:G-inv-thickening}
\end{figure}

\begin{theorem}
\label{thm:Hausdorff}
Let group $G$ act properly and by isometries on a metric space $M$.
For any $G$-invariant subsets $X,Y\subseteq M$, we have $d_\h(X/G, Y/G) = d_\h(X, Y)$.
\end{theorem}

\begin{proof}
The inequality $d_\h(X/G, Y/G) \leq d_\h(X, Y)$ follows directly from the definition of $d_{M/G}$; for $x \in X$ and $y \in Y$, it cannot be that $d(x, y) < d([x], [y])$.

Proving $d_\h(X/G, Y/G) \geq d_\h(X, Y)$ is more involved; we will proceed by contradiction.
To begin, let $d_\h(X/G, Y/G) = \varepsilon$ and toward a contradiction assume $\varepsilon < d_\h(X, Y)$.
Then, by definition of the Hausdorff distance (and perhaps after switching the roles of $X$ and $Y$), there exists a point $x\in $ in that is not contained in $Y^\varepsilon$.
We have that $[x] \in (Y/G)^\varepsilon$ and $x \notin Y^\varepsilon$.
Thus there must exist $y \in Y$ such that $d([y], [x]) < \varepsilon < d(y, x)$.
By definition of $d_{M/G}$, this implies there must exist $g \in G$ such that $d(y, g \cdot x) < d(y, x)$.
Let $\kappa = \frac{d(y, x)+d(y, g \cdot x)}{2}$.
It must be that $Y^\kappa$ contains $g \cdot x$ but not $x$.
As $Y$ is $G$-invariant, so is $Y^\kappa$ by Lemma~\ref{lem:G-balls}.
Thus $g^{-1} \cdot Y^\kappa = Y^\kappa$.
But $g^{-1} \cdot (g \cdot x) = x$ which is not in $Y^\kappa$, a contradiction.
So $d_\h(X/G, Y/G) \geq d_\h(X, Y)$, as desired.
\end{proof}

\begin{figure}[htb]
\centering
\includegraphics[width=4.5in]{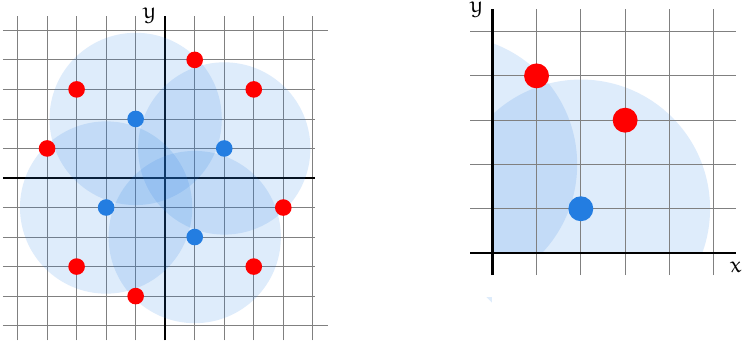}
\caption{A visual example for Theorem~\ref{thm:Hausdorff}, where $G=\Z/4$ acts via rotation by multiples of $90^\circ$ around the origin.}
\label{fig:G-Hausdorff}
\end{figure}

Intuitively, Theorem~\ref{thm:Hausdorff} is a direct consequence of Lemma~\ref{lem:G-balls}.
From the $G$-invariance of the two subsets $X$ and $Y$, we have that operations under $G$ do not affect the relative spacing between our thickened subsets, and thus do not affect the Hausdorff distance between them.
Figure~\ref{fig:G-Hausdorff} provides a visual example.

\section{The Gromov--Hausdorff distance under quotients}
\label{sec:results-GH}

As we have seen, the Hausdorff distance between two $G$-invariant sets is equal to the Hausdorff distance between their orbit spaces.
In this sense, the Hausdorff distance is well-behaved under quotients by proper isometric group actions.
Imagine now if we could move our sets in an attempt to reduce their Hausdorff distance, either before or after quotienting.
Then the relationships we have established become less clear.
The Gromov--Hausdorff distance allows us to investigate these questions.
Is the Gromov--Hausdorff distance between two $G$-invariant sets generally related to the Gromov--Hausdorff distance between their orbit spaces?

The answer is no.
Let $M$ be a metric space, let $G$ act properly by isometries on $M$, and let $X$ and $Y$ be $G$-invariant subsets of $M$.
We begin in Section~\ref{ssec:results-GH-intro} by giving simple examples showing that both of the inequalities $d_\gh(X/G,Y/G) < d_\gh(X,Y)$ and $d_\gh(X/G,Y/G) > d_\gh(X,Y)$ are possible.
We make these examples stronger and more sophisticated in Section~\ref{sec:results-GH-ratio}, showing in Theorem~\ref{thm:gh-ratio} that the ratio $\frac{d_\gh(X/G,M/G)}{d_\gh(X,M)}$ can be made both arbitrarily large and arbitrarily small.

\subsection{Introductory examples}
\label{ssec:results-GH-intro}

In this subsection we provide introductory examples to show that the Gromov--Hausdorff distance is not preserved under group actions.

\begin{example}[$d_\gh(X/G,Y/G) < d_\gh(X,Y)$]
\label{ex:<}
Let $G$ be a group acting properly by isometries on a metric space $M$, and let $X,Y\subseteq M$ be $G$-invariant.
We can find an $M$, $X$, $Y$, and $G$ such that $d_\gh(X/G,Y/G) < d_\gh(X,Y)$.

\begin{figure}
\centering
\begin{subfigure}
\centering
\begin{tikzpicture}[scale=0.7]
\foreach \x in {-1,...,5}
\draw[lightgray] (\x,-3.5) -- (\x,3.5);
\foreach \y in {-3,...,3}
\draw[lightgray] (-1.5,\y) -- (5.5,\y);
\draw[thick] (-1.5,0) -- (5.5,0) node[below] {$x$};
\draw[thick] (0,-3.5) -- (0,3.5) node[left] {$y$};
\foreach \point in {(1,1),(4,1),(1,-1), (4,-1)}
\fill[red] \point circle (6pt);
\foreach \point in {(2,2),(5,2),(2,-2), (5,-2)}
\fill[myblue] \point circle (6pt);
\end{tikzpicture}
\end{subfigure}
\hspace{0.3in}
\begin{subfigure}
\centering
\begin{tikzpicture}[scale=0.7]
\foreach \x in {-1,...,5}
\draw[lightgray] (\x,-3.5) -- (\x,3.5);
\foreach \y in {-3,...,3}
\draw[lightgray] (-1.5,\y) -- (5.5,\y);
\draw[thick] (-1.5,0) -- (5.5,0) node[below] {$x$};
\draw[thick] (0,-3.5) -- (0,3.5) node[left] {$y$};
\foreach \point in {(1,1),(4,1)}
\fill[red] \point circle (6pt);
\foreach \point in {(2,2),(5,2)}
\fill[myblue] \point circle (6pt);
\end{tikzpicture}
\end{subfigure}
\caption{(Left) The sets $X$ (in red) and $Y$ (in blue).
(Right) The sets $X/G$ (in red) and $Y/G$ (in blue).}
\label{fig:counter}
\end{figure}

Let $M = \R^2$, and let $G=\Z/2$ be the group generated by the element which reflects points in $M$ over the $x$-axis.
So $G$ acts by isometries on $M$.
Take $X$ to be the set of four red points and $Y$ to be the set of four blue points in Figure~\ref{fig:counter}(right).
These sets are clearly $G$-invariant and the sets $X/G$ and $Y/G$ are shown in Figure~\ref{fig:counter}(left).
The set $X/G$ can be isometrically embedded into $\R^2$ via a translation of $+1$ units along the x-axis and of $+1$ units along the y-axis, which is an isometry.
Since the translated copies of $X/G$ and $Y/G$ would lie exactly on top of one another, we have $d_\gh(X/G,Y/G) = 0$.
On the other hand, we have that $\diam(X) = \sqrt{13}$ and $\diam(Y) = 5$.
And so we have $d_\gh(X,Y) \ge \frac{1}{2}|\diam(X) - \diam(Y)| = \frac{5}{2}\sqrt{13} > 0.$
Hence $d_\gh(X/G,Y/G) < d_\gh(X,Y)$.
\end{example}

\begin{example}[$d_\gh(X/G,Y/G) > d_\gh(X,Y)$]
\label{ex:>}
Let $G$ be a group acting properly by isometries on a metric space $M$ and $X,Y\subseteq M$ be $G$-invariant.
We can find an $M$, $X$, $Y$, and $G$ such that $d_\gh(X/G,Y/G) > d_\gh(X,Y)$.

\begin{figure}[htb]
\centering
\begin{subfigure}
\centering
\begin{tikzpicture}[scale=.6]
\foreach \x in {-5,...,6}
\draw[lightgray] (\x,-4.5) -- (\x,4.5);
\foreach \y in {-4,...,4}
\draw[lightgray] (-5.5,\y) -- (6.5,\y);
\draw[thick] (-5.5,0) -- (6.5,0) node[below] {$x$};
\draw[thick] (0,-4.5) -- (0,4.5) node[left] {$y$};
\foreach \point in {(-4,3),(-2,3),(-4,1),(-2,1),(-4,-3),(-2,-3),(-4,-1),(-2,-1)}
\fill[red] \point circle (6pt);
\foreach \point in {(0,1),(0,-1),(2,1),(2,-1),(4,1),(4,-1),(6,1),(6,-1)}
\fill[myblue] \point circle (6pt);
\end{tikzpicture}
\end{subfigure}
\hspace{0.1in}
\begin{subfigure}
\centering
\begin{tikzpicture}[scale=.6]
\foreach \x in {-5,...,6}
\draw[lightgray] (\x,-4.5) -- (\x,4.5);
\foreach \y in {-4,...,4}
\draw[lightgray] (-5.5,\y) -- (6.5,\y);
\draw[thick] (-5.5,0) -- (6.5,0) node[below] {$x$};
\draw[thick] (0,-4.5) -- (0,4.5) node[left] {$y$};
\foreach \point in {(-4,3),(-2,3),(-4,1),(-2,1)}
\fill[red] \point circle (6pt);
\foreach \point in {(0,1),(2,1),(4,1),(6,1)}
\fill[myblue] \point circle (6pt);
\end{tikzpicture}
\end{subfigure}
\caption{(Left) The sets $X$ (in red) and $Y$ (in blue).
(right) The sets $X/G$ (in red) and $Y/G$ (in blue).}
\label{fig:counter2}
\end{figure}

Let $M = \R^2$, and let $G=\Z/2$ be a group generated by the element which reflects points in $M$ over the $x$-axis.
Then $G$ acts by isometries on $M$.
Take $X$ to be the set of eight red points and $Y$ to be the set of eight blue points in Figure~\ref{fig:counter2}(left).
These sets are clearly $G$-invariant and the sets $X/G$ and $Y/G$ are shown in Figure~\ref{fig:counter2}(right).
The set $X$ can be re-embedded into $M$ via a $90^{\circ}$ rotation about the center of $X$ (the point $(-3, 0)$) and then a translation of $+6$ units along the x-axis.
This map -- the composition of a rotation and a translation in the plane -- is clearly an isometry.
But then, the sets $X$ and $Y$ would lie exactly on top of one another.
Therefore, $d_\gh(X,Y) = 0$.
On the other hand, we have that $\diam(X/G) = 2\sqrt{2}$ and $\diam(Y/G) = 6$.
Therefore, $d_\gh(X/G,Y/G) \ge \frac{1}{2}|\diam(X/G) - \diam(Y/G)| = 3 - \sqrt{2}$.
And so it follows that $d_\gh(X,Y) < d_\gh(X/G,Y/G)$.

\end{example}

We see immediately that the Gromov--Hausdorff distance is not preserved under the same conditions that we used in Theorem~\ref{thm:Hausdorff}.
In fact, there is no multiplicative bound either, since in each of Examples~\ref{ex:<} and~\ref{ex:>}, one of the Gromov--Hausdorff distances is zero.

We will now place the additional restriction that $X\subseteq Y$, or in other words, that $Y=M$ with $X\subseteq M$.
We show that this subset relationship still does not control the relationship between $d_\gh(X,M)$ and $d_\gh(X/G,M/G)$, via the following examples.

\begin{example}[$d_\gh(X/G,M/G) < d_\gh(X,M)$]
\label{ex:GH-Hausdorff_a}
Let $G$ be a group acting properly by isometries on a metric space $M$, and let $X\subseteq M$ be $G$-invariant.
We can find an $M$, $X$, and $G$ such that $d_\gh(X/G,M/G) < d_\gh(X,M)$.

\begin{figure}[htb]
\centering
\includegraphics[width=4in]{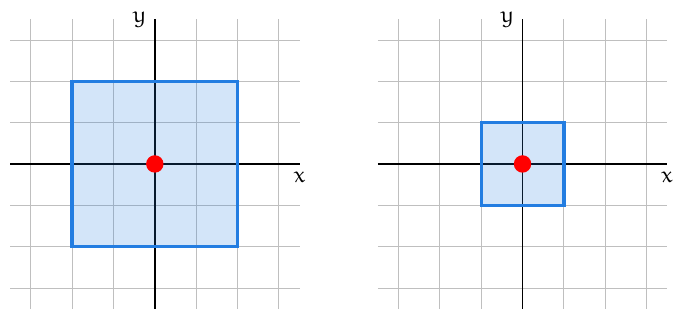}
\caption{The sets $M$ (in blue) and $X$ (in red) are pictured on the left.
Pictured on the right are the sets $M/G$ re-embedded into $\R^{2}$ (in blue) and $X/G$ (in red).}
\label{fig:q3counterEx}
\end{figure}

Let $M=[-2,2]\times [-2,2]$ and let $X=\{(0,0)\}$.
Let $G= \{1,r,r^2,r^3\}$ be the subgroup of rotations inside the dihedral group $D_8$, where $r$ acts as the rotation about the point (0,0) by $\frac{\pi}{2}$ radians.
Then $X$ and $X/G$ are singletons.
By~\cite[Example~6.29]{tuzhilin2020lectures} we can calculate the Gromov--Hausdorff distance between a set and a singleton.
We perform these calculations to obtain the inequality
\[d_\gh(X/G, M/G) = \frac{\diam(M/G)}{2}={\sqrt{2}} < 2\sqrt{2} = \frac{\sqrt{32}}{2} = \frac{\diam(M)}{2}=d_\gh(X,M).\]
\end{example}

So quotienting our space and subset by a group action may decrease the Gromov--Hausdorff distance.

\begin{remark}
Finding an example with $d_\gh(X/G,M/G)<\frac{1}{2}d_\gh(X,M)$ would require the $G$-invariant subset $X$ to satisfy $|X|>1$.
Indeed, if $X=\{x\}$ were a singleton, then we'd have $d_\gh(X/G,M/G)=\frac{1}{2}\diam(M/G)$ and $d_\gh(X,M)=\frac{1}{2}\diam(M)$.
And $\diam(M/G) \ge \frac{1}{2}\diam(M)$ since for all $m,m'\in M$, we have \begin{align*}
d(m,m') &\le d(m,x)+d(x,m') \\
&= d(g\cdot m,x)+d(x,g'\cdot m') \text{ for all }g,g'\in G\text{, since $G$ acts by isometries with $X$ fixed} \\
&=d_{M/G}([m],[x])+d_{M/G}([x],[m'])\\
&\le 2\cdot\diam(M/G).
\end{align*}
\end{remark}

We now provide a second example to that quotienting by a group action may increase the Gromov--Hausdorff distance.

\begin{example}[$d_\gh(X/G,M/G) > d_\gh(X,M)$]
\label{ex:GH-Hausdorff_b}
Let $G$ be a group acting properly by isometries on a metric space $M$ and $X\subseteq M$ be $G$-invariant.
We can find an $M$, $X$, and $G$ such that $d_\gh(X/G,M/G) > d_\gh(X,M)$.

Let $G = (\Z,+)$, and let $M = \Z$.
We let $d\colon \Z\times \Z\rightarrow [0,\infty)$ be the discrete metric where $d(m,n) = 0$ if $m=n$ and $d(m,n) = 1$ if $m\not=n$.
     
Define a map $G\times M \rightarrow M$ by $(g,m)\mapsto 2g+m$.
This is an action of a group on itself.
Indeed, note that for any $g_1,g_2 \in G$ and $m\in M$ we have $g_1\cdot(g_2\cdot m) = g_1\cdot(2g_2 + m) = 2g_1 +2g_2+m = 2(g_1 + g_2) + m = (g_1+g_2)\cdot m$ and $0\cdot 
 m = 0 + m = m$.

The group $G$ acts properly because for every $m\in M$ we have $B_{1/2}(m)=\{m\}$, and the only element $g\in G$ such that $g\cdot B_{1/2}(m)\cap B_{1/2}(m)\neq\emptyset$ is $g=0$.
The action is isometric because it acts by translations.

Now consider the G-invariant set $X=2\Z$.
Since $G$ acts transitively on $2\Z$ we have $X/G = \{[0]\}$.
Notice also that $M/G = \{[0],[1]\}$.
Then $d_{M/G}([0],[1]) = \inf_{g\in G}d(g\cdot 0, 1)  = 1$ since $g\cdot 0=2g$ for any $g\in G$.
Since $X/G$ is a singleton we know that
\[d_\gh(X/G, M/G) = \tfrac{1}{2}\diam(M/G) = \tfrac{1}{2}.\]
Since we have the discrete metric on $\Z$, the map $\phi\colon \Z\rightarrow \Z$ given by $\phi(m)= 2m$ is an isometry.
We have $d_\gh(2\Z,\Z) \leq d_\h^{\Z}(2\Z,\phi(\Z)) = 0$.
So
\[d_\gh(X, M) = 0 < \tfrac{1}{2}= d_\gh(X/G, M/G).\]

Here we leverage the fact that $M$ is isometric to a proper subset of itself to show that the Gromov--Hausdorff distance between $M$ and its subset $X$ is $0$.
This is not possible for compact metric spaces: a compact metric space cannot be isometric to a proper subset of itself~\cite[Theorem~1.6.14]{BuragoBuragoIvanov}.
In the next section we will place more restrictions on our problem so that a reduction like this is not possible.
\end{example}

\subsection{The Gromov--Hausdorff distance for subsets under quotients}
\label{sec:results-GH-ratio}

We now show that for a finite metric space $M$, a group $G$ which acts properly by isometries on $M$, and a $G$-invariant subset $X\subseteq M$, the ratio $\frac{d_\gh(X/G,M/G)}{d_\gh(X,M)}$ can be made both arbitrarily large and arbitrarily small.
This means that the relationship $d_\h(X/G, Y/G) = d_\h(X, Y)$ we showed for the Hausdorff distance in Theorem~\ref{thm:Hausdorff} does not have even an \emph{approximate} version with the Gromov--Hausdorff distance, in general.

\begin{theorem}
\label{thm:gh-ratio}
\ 
\begin{enumerate}
\item[(a)] For any positive integer $n$, there is a finite metric space $M$ with a proper isometric $G$ action and a $G$-invariant $X\subseteq M$ with $d_\gh(X/G, M/G) \leq \frac{1}{n}d_\gh(X, M)$.
\item[(b)] For any positive integer $n$, there is a finite metric space $M$ with a proper isometric $G$ action and a $G$-invariant $X\subseteq M$ with $d_\gh(X/G, M/G) \geq n\cdot d_\gh(X, M)$.
\end{enumerate}
\end{theorem}

\begin{proof}[Proof of Theorem~\ref{thm:gh-ratio}(a)]
We let $n\ge 2$, since it suffices to prove the result for any integer $n\ge 2$.
We will find a finite metric space $M$ with a proper isometric action by a finite group $G$ and a $G$-invariant subset $X$ such that $d_\gh(X/G, M/G) \leq \frac{1}{n}d_\gh(X, M)$.

Let $M = \{r, x_1, x_2, \cdots , x_{n-1}, y_1, y_2, \cdots y_{n+1}\}$, where we define the distances by
\begin{align*}
d(r, x_i) &= i &&\text{for all }1\leq i \leq n-1,\\
d(r, y_i) &= n &&\text{for all }1\leq i \leq n+1,\\
d(x_i, y_j) &= i+n &&\text{for all }1\leq i \leq n-1\text{ and }1\leq j \leq n+1,\\
d(x_i, x_j) &= i + j &&\text{for all }1\leq i \leq n-1\text{ and }1\leq j \leq n-1,\text{ and}\\
d(y_i, y_j) &= 2n &&\text{for all }1\leq i \leq n+1\text{ and }1\leq j \leq n+1.
\end{align*}
As shown in Figure~\ref{fig:ratio-a}, $M$ can be viewed as the vertex set of a tree with a single root $r$ and $2n$ vertices.

\begin{figure}
    \centering%
    \begin{tikzpicture}
        \def\xr{1.5} 
        \def\xrd{0.5} 
        
        \def\ellipAr{2.4} 
        \def\ellipArd{0.05} 
        \def\ellipAt{13} 
        \def\ellipAtd{-7} 
        
        \def\ellipBr{3.1} 
        \def\ellipBrd{0} 
        \def\ellipBt{-110} 
        \def\ellipBtd{-7} 

        \foreach \i in {0,1,2} {
            \draw[fill] ({(\ellipAr+\i*\ellipArd)*cos(\ellipAt+\i*\ellipAtd)},{(\ellipAr+\i*\ellipArd)*sin(\ellipAt+\i*\ellipAtd)}) circle (1pt);
        }

        \foreach \i in {0,1,2} {
            \draw[fill] ({(\ellipBr+\i*\ellipBrd)*cos(\ellipBt+\i*\ellipBtd)},{(\ellipBr+\i*\ellipBrd)*sin(\ellipBt+\i*\ellipBtd)}) circle (1pt);
        }
        
        \foreach \i in {0,1,2} {
            \draw[line width=1pt] (0,0) -- ({(\xr+\i*\xrd)*sin(\i*30)},{(\xr+\i*\xrd)*cos(\i*30)});
        }
        \draw[line width=1pt] (0,0) -- ({(\xr+3*\xrd)*sin(120)},{(\xr+3*\xrd)*cos(120)});
        \foreach \i in {5,6,8} {
            \draw[line width=1pt] (0,0) -- ({(\xr+4*\xrd)*sin(\i*30)},{(\xr+4*\xrd)*cos(\i*30)});
        }

        \node[left] at ({(\xr)*sin(0)/2},{(\xr)*cos(0)/2}) {$1$};
        \node[above left] at ({(\xr+\xrd)*sin(30)/2},{(\xr+\xrd)*cos(30)/2}) {$2$};
        \node[above left] at ({(\xr+2*\xrd)*sin(60)/2},{(\xr+2*\xrd)*cos(60)/2}) {$3$};
        \node[above right] at ({(\xr+3*\xrd)*sin(120)/2},{(\xr+3*\xrd)*cos(120)/2}) {$n-1$};
        \node[right] at ({(\xr+4*\xrd)*sin(150)/2},{(\xr+4*\xrd)*cos(150)/2}) {$n$};
        \node[right] at ({(\xr+4*\xrd)*sin(180)/2},{(\xr+4*\xrd)*cos(180)/2}) {$n$};
        \node[below right] at ({(\xr+4*\xrd)*sin(240)/2},{(\xr+4*\xrd)*cos(240)/2}) {$n$};
    
        \draw[fill] (0,0) circle (2pt) node[above left] {$r$};
        \draw[fill] ({(\xr)*sin(0)},{(\xr)*cos(0)}) circle (2pt) node[above] {$x_1$};
        \draw[fill] ({(\xr+\xrd)*sin(30)},{(\xr+\xrd)*cos(30)}) circle (2pt) node[above] {$x_2$};
        \draw[fill] ({(\xr+2*\xrd)*sin(60)},{(\xr+2*\xrd)*cos(60)}) circle (2pt) node[above] {$x_3$};
        \draw[fill] ({(\xr+3*\xrd)*sin(120)},{(\xr+3*\xrd)*cos(120)}) circle (2pt) node[right] {$x_{n-1}$};
        \draw[fill] ({(\xr+4*\xrd)*sin(150)},{(\xr+4*\xrd)*cos(150)}) circle (2pt) node[below right] {$y_1$};
        \draw[fill] ({(\xr+4*\xrd)*sin(180)},{(\xr+4*\xrd)*cos(180)}) circle (2pt) node[below] {$y_2$};
        \draw[fill] ({(\xr+4*\xrd)*sin(240)},{(\xr+4*\xrd)*cos(240)}) circle (2pt) node[below] {$y_{n+1}$};
    \end{tikzpicture}
\hspace{20mm}
    \begin{tikzpicture}
        \def\xr{1.5} 
        \def\xrd{0.5} 
        
        \def\ellipAr{2.4} 
        \def\ellipArd{0.05} 
        \def\ellipAt{13} 
        \def\ellipAtd{-7} 
        
        \def\ellipBr{3.1} 
        \def\ellipBrd{0} 
        \def\ellipBt{-110} 
        \def\ellipBtd{-7} 

        \foreach \i in {0,1,2} {
            \draw[fill] ({(\ellipAr+\i*\ellipArd)*cos(\ellipAt+\i*\ellipAtd)},{(\ellipAr+\i*\ellipArd)*sin(\ellipAt+\i*\ellipAtd)}) circle (1pt);
        }
        
        \foreach \i in {0,1,2} {
            \draw[line width=1pt] (0,0) -- ({(\xr+\i*\xrd)*sin(\i*30)},{(\xr+\i*\xrd)*cos(\i*30)});
        }
        \draw[line width=1pt] (0,0) -- ({(\xr+3*\xrd)*sin(120)},{(\xr+3*\xrd)*cos(120)});
        \foreach \i in {5} {
            \draw[line width=1pt] (0,0) -- ({(\xr+4*\xrd)*sin(\i*30)},{(\xr+4*\xrd)*cos(\i*30)});
        }

        \node[left] at ({(\xr)*sin(0)/2},{(\xr)*cos(0)/2}) {$1$};
        \node[above left] at ({(\xr+\xrd)*sin(30)/2},{(\xr+\xrd)*cos(30)/2}) {$2$};
        \node[above left] at ({(\xr+2*\xrd)*sin(60)/2},{(\xr+2*\xrd)*cos(60)/2}) {$3$};
        \node[above right] at ({(\xr+3*\xrd)*sin(120)/2},{(\xr+3*\xrd)*cos(120)/2}) {$n-1$};
        \node[right] at ({(\xr+4*\xrd)*sin(150)/2},{(\xr+4*\xrd)*cos(150)/2}) {$n$};
    
        \draw[fill] (0,0) circle (2pt) node[above left] {$[r]$};
        \draw[fill] ({(\xr)*sin(0)},{(\xr)*cos(0)}) circle (2pt) node[above] {$[x_1]$};
        \draw[fill] ({(\xr+\xrd)*sin(30)},{(\xr+\xrd)*cos(30)}) circle (2pt) node[above] {$[x_2]$};
        \draw[fill] ({(\xr+2*\xrd)*sin(60)},{(\xr+2*\xrd)*cos(60)}) circle (2pt) node[above] {$[x_3]$};
        \draw[fill] ({(\xr+3*\xrd)*sin(120)},{(\xr+3*\xrd)*cos(120)}) circle (2pt) node[right] {$[x_{n-1}]$};
        \draw[fill] ({(\xr+4*\xrd)*sin(150)},{(\xr+4*\xrd)*cos(150)}) circle (2pt) node[below right] {$[y_1]=\ldots=[y_{n+1}]$};
    \end{tikzpicture}    
\caption{Proof of Theorem~\ref{thm:gh-ratio}(a).
(Left) Space $M$.
(Right) Space $M/G$.}
\label{fig:ratio-a}
\end{figure}
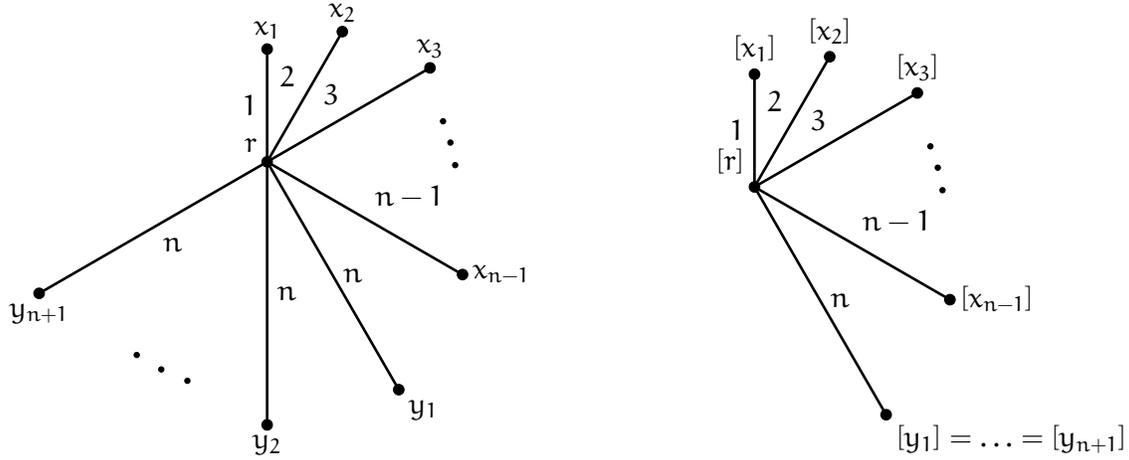

Let $G = S_{n+1}$ be the symmetric group on $n+1$ elements.
Let $G$ act on $M$ by permuting the elements $\{y_1, \cdots y_{n+1}\}$ according to their subscripts, keeping the remaining elements in $M$ fixed.
Explicitly, for $\sigma\in S_{n+1}$, we set $\sigma\cdot y_i=y_{\sigma(i)}$ for $1\le i\le n+1$, we set $\sigma\cdot r=r$, and we set $\sigma\cdot x_i=x_i$ for $1\le i\le n-1$.
Let $X$ be the $G$-invariant set $\{r, x_1, x_2, \cdots , x_{n-1}\}$, which remains fixed under the group action.

We will show that $d_\gh(X, M)  \geq n$.
Observe that $|X|=n$, and that $M$ has $n + 1$ points $y_1,\ldots,y_{n+1}$ with $d(y_i,y_j)=2n$ for all $i\neq j$.
Therefore, for any correspondence $C\subseteq X\times M$, there is some point $x \in X$ which is corresponded with two unique points, say $y_i, y_j \in M$, with $d(y_i,y_j)=2n$.
From Definition~\ref{def:AdditiveDistortion} we can get a lower bound for our distortion by 
\[\dis(\mathbf{C})\coloneqq\sup_{(x_1,m_1), (x_2,m_2) \in \textbf{C}} |d_X(x_1,x_2) - d_M(m_1,m_2)| \ge |d_X(x,x) - d_M(y_i,y_j)| = |0 - 2n| = 2n.\]
Since this is a lower bound for the distortion of every correspondence, using Theorem~\ref{thm:distortion_GH} we have $d_\gh(X, M) = \frac{1}{2}\inf_{R}(\dis(R)) \geq \frac{1}{2}\dis(C) \geq n$.

Now we show that $d_\gh(X/G,M/G) \leq 1$.
For notational convenience, define $x_0\coloneqq r$ and $x_n\coloneqq y_1$.
Notice that $X/G$ and $X$ are isometric, and that $M/G$ and $X\cup\{x_n\}$ are isometric.
Consider the correspondence $R\subseteq X/G\times M/G$ consisting of $([x_0],[x_0])$ and of $([x_i],[x_{i+1}])$ for each $0\leq i\leq n-1$.
We have \begin{align*}
|d_{X/G}([x_0],[x_i])-d_{M/G}([x_0],[x_{i+1}])| & =|i-(i+1)|=1, \\
|d_{X/G}([x_i],[x_j])-d_{M/G}([x_{i+1}],[x_{j+1}])| & =|(i+j)-(i+j+2)|=2
\end{align*} for all $0\leq i\leq n-1$ and $0\leq j\leq n-1$.
Hence, $\dis(R)=2$ and so $d_\gh(X/G,M/G)\leq1$.

Therefore $d_\gh(X/G,M/G)\le \frac{1}{n}\cdot n\le \frac{1}{n}\cdot d_\gh(X,M)$, as desired.
\end{proof}

\begin{figure}[htb]
\centering%
\begin{tikzpicture}[scale = 1]
    \def\xr{2.4} 
    \def\xrd{0.5} 
    \def\thetad{75} 
    \def\thetar{40} 
    \def\lenr{0.7} 
    \def\labelr{0.25} 
    \def\labelrp{0.6} 
    \def\elliptd{6} 
    \def\ellipr{1.0} 
    \def\gellipr{3.0} 
    \def\gellipt{145} 
    \def\gelliptd{8} 

    \newcommand{\drawgroup}[7]{%
        \foreach \j in {0,1,2} {
            \draw[fill] ({(#1)*\ellipr*sin(#2+(#3)/2+(\j-1)*\elliptd)},{(#1)*\ellipr*cos(#2+(#3)/2+(\j-1)*\elliptd))}) circle (1pt);
        }

        \coordinate (A) at ({(#1)*sin(#2)},{(#1)*cos(#2)});
        \coordinate (B) at ({(#1)*sin(#2+#3)},{(#1)*cos(#2+#3)});
        
        \draw[line width=1pt] (0,0) -- (A);
        \draw[line width=1pt] (0,0) -- (B);
            
        \node at ({(#1)*sin(#2)*\lenr-(#5)*cos(#2))},{(#1)*cos(#2)*\lenr+(#5)*sin(#2)}) {#4};
        \node at ({(#1)*sin(#2+#3)*\lenr-(#5)*cos(#2+#3))},{(#1)*cos(#2+#3)*\lenr+(#5)*sin(#2+#3)}) {#4};
            
        \node at ({#1+1.2+\labelrp)*sin(#2+#3/2))},{(#1+1.2+\labelrp)*cos(#2+#3/2))}) {$n+1$};

        \node at ({(#1+\labelrp)*sin(#2)},{(#1+\labelrp)*cos(#2)}) {#6};
        \node at ({(#1+\labelrp)*sin(#2+#3)},{(#1+\labelrp)*cos(#2+#3)}) {#7};
            
        \draw[fill] (A) circle (2pt);
        \draw[fill] (B) circle (2pt);

        \draw[line width=1.5pt,decorate,decoration={calligraphic brace,amplitude=6pt,raise=32pt}] (A) -- (B);
    }

    \foreach \j in {0,1,2} {
        \draw[fill] ({\gellipr*sin(\gellipt+(\j-1)*\gelliptd)},{\gellipr*cos(\gellipt+(\j-1)*\gelliptd))}) circle (1.25pt);
    }

    \drawgroup{\xr}{0}{\thetar}{$1$}{\labelr}{$a_{1}^{1}$}{$a_{1}^{n+1}$}
    \drawgroup{\xr+\xrd}{75}{\thetar}{$2$}{\labelr}{$a_{2}^{1}$}{$a_{2}^{n+1}$}
    \drawgroup{\xr+2*\xrd}{180}{\thetar}{$n-1$}{0.55}{$a_{n-1}^{1}$}{$a_{n-1}^{n+1}$}
    \drawgroup{\xr+3*\xrd}{255}{\thetar}{$n$}{\labelr}{$a_{n}^{1}$}{$a_{n}^{n+1}$}
        
    \draw[fill] (0,0) circle (2pt) node[above left,xshift=0,yshift=4] {$r$};
\end{tikzpicture}
\vskip 5mm
\begin{tikzpicture}[scale = 1]
    \def\xr{1.5} 
    \def\xrd{0.5} 
        
    \def\ellipAr{2.4} 
    \def\ellipArd{0.05} 
    \def\ellipAt{13} 
    \def\ellipAtd{-7} 
        
    \def\ellipBr{3.1} 
    \def\ellipBrd{0} 
    \def\ellipBt{-110} 
    \def\ellipBtd{-7} 

    \foreach \i in {0,1,2} {
        \draw[fill] ({(\ellipAr+\i*\ellipArd)*cos(\ellipAt+\i*\ellipAtd)},{(\ellipAr+\i*\ellipArd)*sin(\ellipAt+\i*\ellipAtd)}) circle (1pt);
    }

    \foreach \i in {0,1,2} {
        \draw[fill] ({(\ellipBr+\i*\ellipBrd)*cos(\ellipBt+\i*\ellipBtd)},{(\ellipBr+\i*\ellipBrd)*sin(\ellipBt+\i*\ellipBtd)}) circle (1pt);
    }
        
    \foreach \i in {0,1,2} {
        \draw[line width=1pt] (0,0) -- ({(\xr+\i*\xrd)*sin(\i*30)},{(\xr+\i*\xrd)*cos(\i*30)});
    }
    \draw[line width=1pt] (0,0) -- ({(\xr+3*\xrd)*sin(120)},{(\xr+3*\xrd)*cos(120)});
    \foreach \i in {5,6,8} {
        \draw[line width=1pt] (0,0) -- ({(\xr+4*\xrd)*sin(\i*30)},{(\xr+4*\xrd)*cos(\i*30)});
    }

    \node[left] at ({(\xr)*sin(0)/2},{(\xr)*cos(0)/2}) {$1$};
    \node[above, xshift = -3pt] at ({(\xr+\xrd)*sin(30)/2},{(\xr+\xrd)*cos(30)/2}) {$2$};
    \node[above] at ({(\xr+2*\xrd)*sin(60)/2},{(\xr+2*\xrd)*cos(60)/2}) {$3$};
    \node[above right] at ({(\xr+3*\xrd)*sin(120)/2},{(\xr+3*\xrd)*cos(120)/2}) {$n-1$};
    \node[right] at ({(\xr+4*\xrd)*sin(150)/2},{(\xr+4*\xrd)*cos(150)/2}) {$n$};
    \node[right] at ({(\xr+4*\xrd)*sin(180)/2},{(\xr+4*\xrd)*cos(180)/2}) {$n$};
    \node[below right] at ({(\xr+4*\xrd)*sin(240)/2},{(\xr+4*\xrd)*cos(240)/2}) {$n$};
    
    \draw[fill] (0,0) circle (2pt) node[above left] {$r$};
    \draw[fill] ({(\xr)*sin(0)},{(\xr)*cos(0)}) circle (2pt) node[above] {$[a_{1}^{j}]$};
    \draw[fill] ({(\xr+\xrd)*sin(30)},{(\xr+\xrd)*cos(30)}) circle (2pt) node[above] {$[a_{2}^{j}]$};
    \draw[fill] ({(\xr+2*\xrd)*sin(60)},{(\xr+2*\xrd)*cos(60)}) circle (2pt) node[above] {$[a_{3}^{j}]$};
    \draw[fill] ({(\xr+3*\xrd)*sin(120)},{(\xr+3*\xrd)*cos(120)}) circle (2pt) node[right] {$[a_{n-1}^{j}]$};
    \draw[fill] ({(\xr+4*\xrd)*sin(150)},{(\xr+4*\xrd)*cos(150)}) circle (2pt) node[below right] {$[a_{n}^{1}]$};
    \draw[fill] ({(\xr+4*\xrd)*sin(180)},{(\xr+4*\xrd)*cos(180)}) circle (2pt) node[below] {$[a_{n}^{2}]$};
    \draw[fill] ({(\xr+4*\xrd)*sin(240)},{(\xr+4*\xrd)*cos(240)}) circle (2pt) node[below] {$[a_{n}^{n+1}]$};
\end{tikzpicture}
    
\caption{
Proof of Theorem~\ref{thm:gh-ratio}(b).
(Top) Space $M$.
(Bottom) Space $M/G$.
For $1\le i\le n-1$, we write $[a_i^j]$ to denote the equivalence class $[a_i^1]=[a_i^2]=\ldots=[a_i^{n+1}]$.
}
\label{fig:ratio-b}
\end{figure}
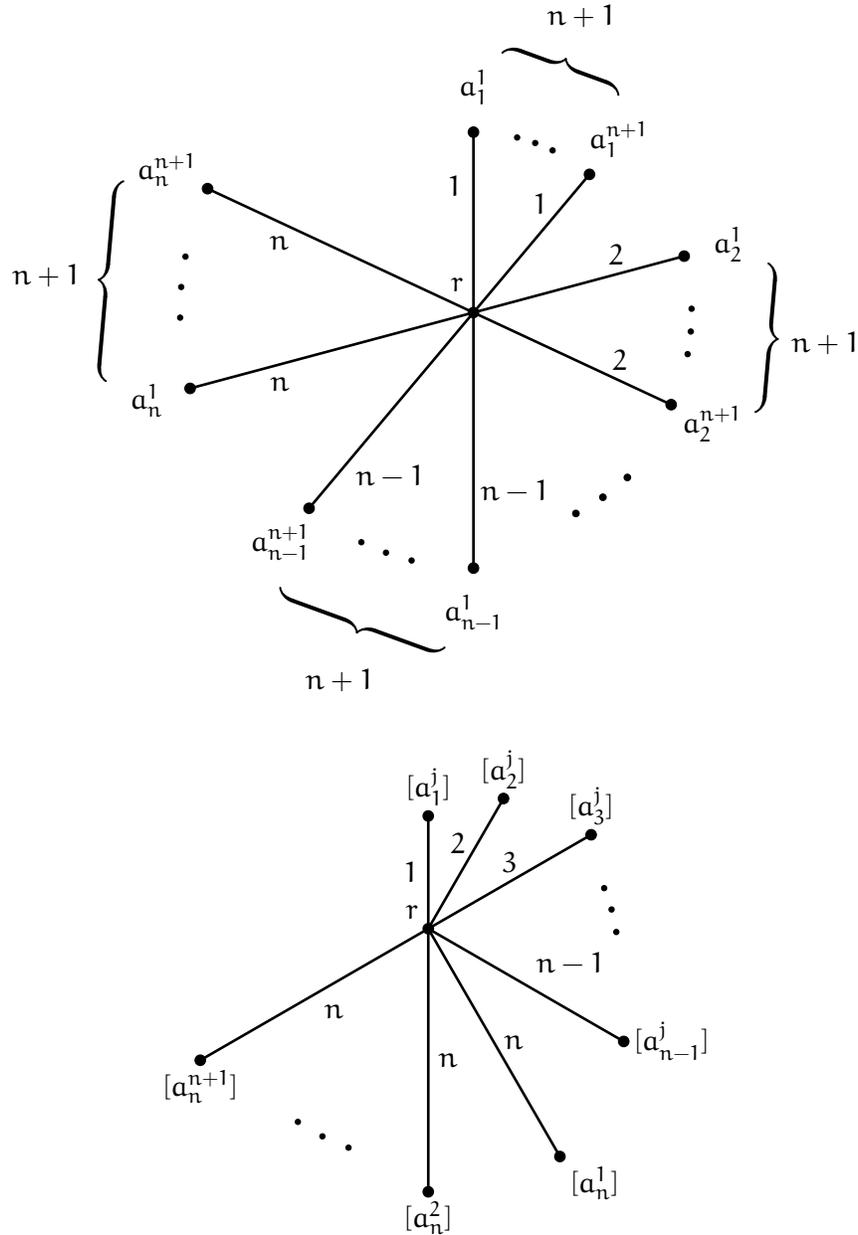

\begin{proof}[Proof of Theorem~\ref{thm:gh-ratio}(b)]
For all natural numbers $n \geq 2$, we will find a finite metric space $M$ with a proper isometric action by a finite group $G$ and a $G$-invariant subset $X$ such that $d_\gh(X/G, M/G) \geq n\cdot d_\gh(X, M)$.

Let $M = \{r\} \cup \{a_i^j \mid 1\leq i \leq n, 1 \leq j \leq n+1\}$; see Figure~\ref{fig:ratio-b}.
We define the distances by
\[
d(r, a_i^j) = i \quad\text{and}\quad d(a_i^j, a_{i'}^{j'})=i+i' \quad\text{for all}\quad 1\leq i,i' \leq n \quad\text{and}\quad 1 \leq j \leq n+1.
\]
Note $M$ can be viewed as the root and leaf vertices of the tree in Figure~\ref{fig:ratio-b}.

Let $G = S_{n+1}$ be the symmetric group on $n+1$ elements.
It acts on $M$ as follows.
For $\sigma \in G$, we set $\sigma \cdot r = r$, and we set
\[\sigma \cdot a_i^j =  
\begin{dcases*}
a_{i}^{\sigma(j)} & if $i \leq n-1$\\
a_i^j & if $i = n$.
\end{dcases*}
\]
We let $X =  \{r\} \cup \{a_i^j \mid 1\leq i \leq n-1, 1 \leq j \leq n+1\}$.
Note $G\cdot X=X$ and so $X$ is indeed $G$-invariant.

Now we show that $d_\gh(X,M) \leq 1$.
For convenience, define $a_0^j\coloneqq r$ for each $1\leq j\leq n+1$.
Consider the correspondence $R\subseteq X \times M$ consisting of $(a_0^j,a_0^j)$ for all $1\leq j\leq n+1$ and of $(a_i^j,a_{i+1}^j)$ for all $0\leq i\leq n-1$ and $1\leq j\leq n+1$.
The proof that $\dis(R)=2$ is nearly identical to that of the calculation of the distortion of the correspondence $R$ in the proof of Theorem~\ref{thm:gh-ratio}(a), from which we may conclude that $d_\gh(X,M)\leq1$.

To see that $d_\gh(X/G, M/G) \geq n$, note that $X/G$ and $M/G$ are in fact isometric to the metric spaces ($X$ and $Y$) used in the proof of Theorem~\ref{thm:gh-ratio}(a), and hence by the same argument given there we have $d_\gh(X/G, M/G) \ge n$.

Combining these two bounds, we see that $d_\gh(X/G, M/G) \ge n = n\cdot 1 \ge n\cdot d_\gh(X, M)$, as desired.
\end{proof}

\section{Additional density assumptions}
\label{sec:results-density}

Our Theorem~\ref{thm:gh-ratio} demonstrates that the Gromov--Hausdorff distance is not particularly well-behaved under group actions.
What if we have not only that $X$ is a $G$-invariant subset of $M$, but also that $X$ is arbitrarily dense in $M$ --- can one do better in this setting?

We show the answer is no.
Indeed, in Section~\ref{ssec:results-density} we show that there exists a compact metric space $M$ equipped with a proper isometric $G$ action such that for any $\delta>0$, there is a $G$-invariant subset $X\subseteq M$ satisfying $d_\h(X,M)<\delta$ and $d_\gh(X/G,M/G) > d_\gh(X,M)$.
We generalize this construction in Section~\ref{ssec:results-density-generalization}.
Then, using this generalization, we strengthen the aforementioned result, showing in Section~\ref{ssec:arb-dense-arb-constants} that for any positive integer $n$, there is a compact metric space $M$ with a proper isometric $G$ action such that for any $\delta>0$, there is a $G$-invariant subset $X\subseteq M$ with $d_\h(X,Z)<\delta$ satisfying $\frac{d_\gh(X/G,M/G)}{d_\gh(X,M)}\ge n$.

Most of the examples in this paper so far have been with finite metric spaces.
When we move to the context where we want to allow the subset $X\subseteq M$ to be arbitrarily dense, we will also move to the more complicated setting where $M$ is often infinite, though still often compact.

\subsection{Density alone is not enough}
\label{ssec:results-density}

It turns out that additional density assumptions alone do not suffice.
Indeed, in the following theorem we prove that there is a compact metric space $M$ such that for any $\delta>0$, there is a $G$-invariant subset $X\subseteq M$ that is $\delta$-close (meaning $d_\h(X,M)<\delta$) but that still satisfies $d_\gh(X/G,M/G)>d_\h(X,M)$.

\begin{theorem}
\label{thm:GH-Hausdorff_dense}
There exists a compact metric space $M$ equipped with a proper action of a group $G$ by isometries such that for any $\delta>0$, there is a $G$-invariant subset $X\subseteq M$ satisfying $d_\h(X,M)<\delta$ and $d_\gh(X/G,M/G) > d_\gh(X,M)$.
\end{theorem}

\begin{proof}
To begin, define $M_i=\{a_i,b_i,c_i\}$ to be an ``isosceles triangle'' metric space, with a base of length $d(b_i,c_i)=2^{-i}/10^{10}$, and with two equal sides of length $d(a_i,b_i)=d(a_i,c_i)=2^{-i}$.
So, the points $b_i$ and $c_i$ are on the base of the triangle, and the apex is at $a_i$.
Let the group $G=\Z/2$ act by swapping the bottom points $b_i$ and $c_i$ and mapping the point $a_i$ to itself.
Let $X_i\subseteq M_i$ be the $G$-invariant subset $X_i=\{b_i,c_i\}$.

We claim $d_\gh(X_i,M_i) < 2^{-i-1}$ for all $0 \le i < \infty$.
Indeed, consider the isometric embedding of $M_i$ and $X_i$ into the tree metric space in Figure~\ref{fig:dGH-Xi-Mi}.
The Hausdorff distance between these isometrically embedded copies of $X_i$ and of $M_i$ is equal to $\frac{2^{-i}-2^{-i}\cdot 10^{-10}}{2} = 2^{-i-1} (1 - \frac{1}{10^{10}}) < 2^{-i-1}$.
These embeddings are associated with the correspondence $R_i \subseteq X_i \times M_i$ given by $R=\{(b_i,a_i),(c_i,b_i),(c_i,c_i)\}$ with $\dis(R_i)<2^i$.
Also, we note $d_\gh(X_i/G,M_i/G)=\frac{1}{2}\diam(M_i/G)=\frac{1}{2}2^{-i}=2^{-i-1}$ since $X_i/G$ is a singleton.

\begin{figure}
\centering
\begin{tikzpicture} [scale=1]
    \node[] at (-1,2.5) {\Large$M_i=$};
    
    \draw[gray,thick,dashed] (0,0)--(2,5) node[midway, above left, black]{$2^{-i}$};
     \draw[gray,thick,dashed] (2,5)--(4,0) node[midway, above right, black]{$2^{-i}$};
    \draw[gray,thick,dashed] (0,0)--(4,0) node[midway, below, black]{$\frac{2^{-i}}{10^{10}}$};

    \draw[red] (0,0)--(0,0) node[draw, shape = circle, fill = red, minimum size = 0.3cm, inner sep=0pt]{} node[above left, black] {\Large$b_i$};
    \draw[red] (2,5)--(2,5) node[draw, shape = circle, fill = red, minimum size = 0.3cm, inner sep=0pt]{} node[above left, black] {\Large$a_i$};
    \draw[red] (4,0)--(4,0) node[draw, shape = circle, fill = red, minimum size = 0.3cm, inner sep=0pt]{} node[above right, black] {\Large$c_i$};

    \node[] at (-1,-2.5) {\Large$X_i=$};
        
    \draw[gray,thick,dashed] (0,-2.5)--(4,-2.5) node[midway, below, black]{$\frac{2^{-i}}{10^{10}}$};

    \draw[myblue] (0,-2.5)--(0,-2.5) node[draw, shape = circle, fill = myblue, minimum size = 0.3cm, inner sep=0pt]{} node[above, black] {\Large$[b_i]$};
    \draw[myblue] (4,-2.5)--(4,-2.5) node[draw, shape = circle, fill = myblue, minimum size = 0.3cm, inner sep=0pt]{} node[above, black] {\Large$[c_i]$};
\end{tikzpicture}
\hspace{1.5cm}
\begin{tikzpicture} [scale=1]
    \draw[black] (2,4.5)--(2,7.5) node[midway, right, black]{$2^{-i-1} (1 - \frac{1}{10^{10}})$};
    \draw[black] (2,2.5)--(2,4.5) node[midway, right, black]{$\frac{2^{-i}}{10^{10}}$};
    \draw[black] (2,0.5)--(2,2.5) node[midway, right, black]{$2^{-i-1} (1 - \frac{2}{10^{10}})$};
    \draw[black] (0,0)--(2,0.5) node[midway, above left, black]{$\frac{2^{-i-1}}{10^{10}}$};
    \draw[black] (4,0)--(2,0.5) node[midway, above right, black]{$\frac{2^{-i-1}}{10^{10}}$};

    \draw[red] (0,0)--(0,0) node[draw, shape = circle, fill = red, minimum size = 0.3cm, inner sep=0pt]{} node[below left, black] {\Large$b_i$};
    \draw[red] (2,7.5)--(2,7.5) node[draw, shape = circle, fill = red, minimum size = 0.3cm, inner sep=0pt]{} node[above left, black]{\Large$a_i$};
    \draw[red] (4,0)--(4,0) node[draw, shape = circle, fill = red, minimum size = 0.3cm, inner sep=0pt]{} node[below right, black] {\Large$c_i$};
    \draw[myblue] (2,2.5)--(2,2.5) node[draw, shape = circle, fill = myblue, minimum size = 0.3cm, inner sep=0pt]{} node[above left, black] {\Large$[c_i]$};
    \draw[myblue] (2,4.5)--(2,4.5) node[draw, shape = circle, fill = myblue, minimum size = 0.3cm, inner sep=0pt]{} node[above left, black] {\Large$[b_i]$};
\end{tikzpicture}
\caption{
(Left) The metric spaces $M_i$ and $X_i$.
(Right) Isometric embeddings of $X_i$ and $M_i$ into a tree metric space, showing $d_\gh(X_i,M_i) < 2^{-i-1}$.
}
\label{fig:dGH-Xi-Mi}
\end{figure}
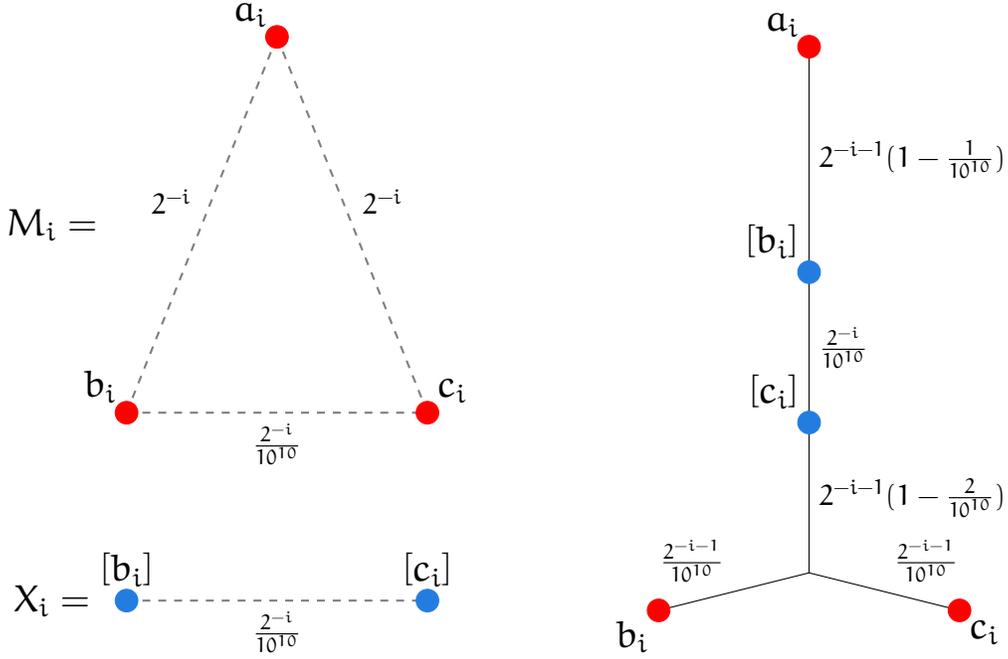

We choose our compact space $M$ to be the union of countably infinitely many of these $M_i$ metric spaces, where we set the distance between any point in $M_i$ and any point in $M_{i+1}$ to be $2^{10-i}$, along with a limit point $z$ so that $M$ is compact; see Figure~\ref{fig:dense-M-and-X2}(top).
In symbols, we define
\[ M=\left(\bigsqcup_{i=0}^{\infty} M_i\right)\sqcup\{z\}.\]
To be explicit, 
\begin{itemize}
\item for any points $m_i \in M_i$ and $m_j\in M_j$ with $i<j$, we define $d(m_i,m_j)=\sum_{k=i}^{j-1}2^{10-k}=2^{11-i}-2^{11-j}$,
\item for any point $m_i\in M_i$ we define $d(m_i,z)=\sum_{k=i}^{\infty}2^{10-k}=2^{11-i}$, and
\item the distance in $M$ between any two points in $M_i$ is the same as their distance in $M_i$.
\end{itemize}
We let the group $G=\Z/2$ act on $M$ by keeping the point $z$ and each point $a_i$ fixed, and by swapping each $b_i$ and $c_i$.

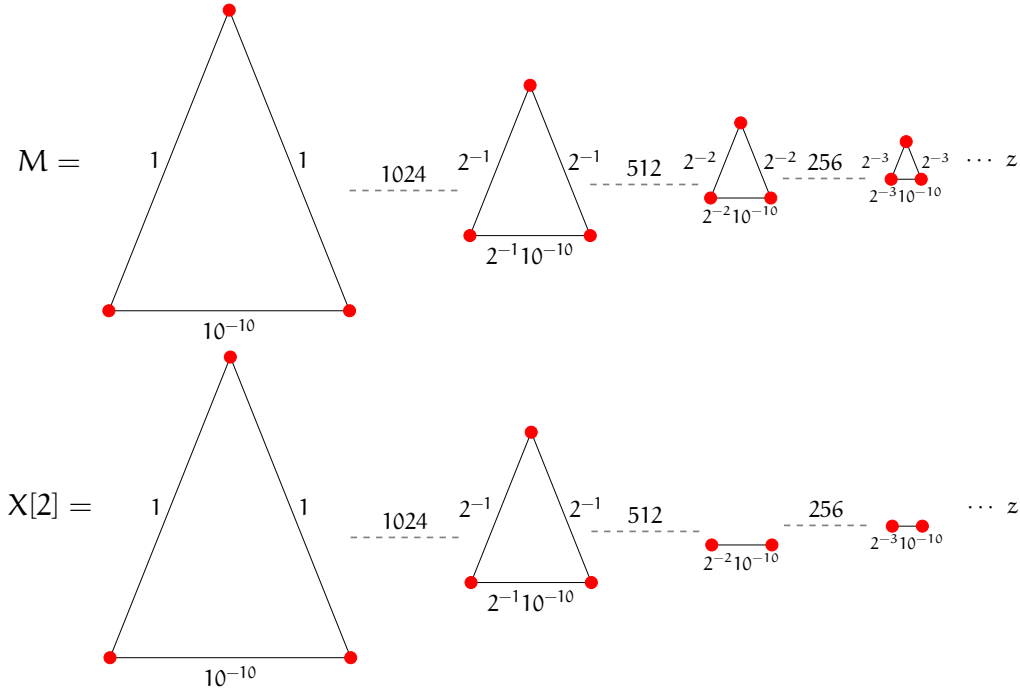
\begin{figure}[ht]
\centering

\scalebox{0.8}{

\begin{tikzpicture}
        
        \node[] at (-1,2.5) {\Large$M=$};

        \draw[black] (0,0)--(2,5) node[]{};
        \draw[black] (2,5)--(4,0) node[]{};
        \draw[black] (0,0)--(4,0) node[]{};

        \draw[red] (0,0)--(0,0) node[draw, shape = circle, fill = red, minimum size = 0.2cm, inner sep=0pt]{};
        \draw[red] (2,5)--(2,5) node[draw, shape = circle, fill = red, minimum size = 0.2cm, inner sep=0pt]{};
        \draw[red] (4,0)--(4,0) node[draw, shape = circle, fill = red, minimum size = 0.2cm, inner sep=0pt]{};

        \node[left] at (1,2.5) {\normalfont$1$};
        \node[right] at (3,2.5) {\normalfont$1$};
        \node[below] at (2,0) {\normalfont$10^{-10}$};
        
        
        \draw[gray,thick,dashed] (4,2)--(5.8,2) node[]{};
        \node[above] at (4.9,2) {\normalfont$1024$};

        \draw[gray,thick,dashed] (8,2.1)--(9.8,2.1) node[]{};
        \node[above] at (8.9,2.1) {\normalfont$512$};

        \draw[gray,thick,dashed] (11.2,2.2)--(12.6,2.2) node[]{};
        \node[above] at (11.9,2.2) {\normalfont$256$};
        
        
        \draw[black] (6,1.25)--(7,3.75) node[]{};
        \draw[black] (7,3.75)--(8,1.25) node[]{};
        \draw[black] (6,1.25)--(8,1.25) node[]{};

        \draw[red] (6,1.25)--(6,1.25) node[draw, shape = circle, fill = red, minimum size = 0.2cm, inner sep=0pt]{};
        \draw[red] (7,3.75)--(7,3.75) node[draw, shape = circle, fill = red, minimum size = 0.2cm, inner sep=0pt]{};
        \draw[red] (8,1.25)--(8,1.25) node[draw, shape = circle, fill = red, minimum size = 0.2cm, inner sep=0pt]{};
        
        \node[left] at (6.5,2.5) {\normalfont$2^{-1}$};
        \node[right] at (7.5,2.5) {\normalfont$2^{-1}$};
        \node[below] at (7,1.25){\normalfont$2^{-1}10^{-10}$};


        \draw[black] (10,1.875)--(11,1.875) node[]{};
        \draw[black] (10,1.875)--(10.5,3.125) node[]{};
        \draw[black] (10.5,3.125)--(11,1.875) node[]{};

        \draw[red] (10,1.875)--(10,1.875) node[draw, shape = circle, fill = red, minimum size = 0.2cm, inner sep=0pt]{};
        \draw[red] (11,1.875)--(11,1.875) node[draw, shape = circle, fill = red, minimum size = 0.2cm, inner sep=0pt]{};
        \draw[red] (10.5,3.125)--(10.5,3.125) node[draw, shape = circle, fill = red, minimum size = 0.2cm, inner sep=0pt]{};
        
        \node[left] at (10.25,2.5) {\small$2^{-2}$};
        \node[right] at (10.75,2.5) {\small$2^{-2}$};
        \node[below] at (10.5,1.875) {\footnotesize$2^{-2}10^{-10}$};


        \draw[black] (13,2.1875)--(13.5,2.1875) node[]{};
        \draw[black] (13,2.1875)--(13.25,2.8125) node[]{};
        \draw[black] (13.5,2.1875)--(13.25,2.8125) node[]{};

        \draw[red] (13,2.1875)--(13,2.1875) node[draw, shape = circle, fill = red, minimum size = 0.2cm, inner sep=0pt]{};
        \draw[red] (13.5,2.1875)--(13.5,2.1875) node[draw, shape = circle, fill = red, minimum size = 0.2cm, inner sep=0pt]{};
        \draw[red] (13.25,2.8125)--(13.25,2.8125) node[draw, shape = circle, fill = red, minimum size = 0.2cm, inner sep=0pt]{};

        \node[left] at (13.125,2.5) {\scriptsize$2^{-3}$};
        \node[right] at (13.375,2.5) {\scriptsize$2^{-3}$};
        \node[below] at (13.25,2.1875) {\scriptsize$2^{-3}10^{-10}$};

        \node[] at (14.5,2.5) {\normalsize$\cdots$};
        \node[] at (15,2.5) {\normalsize$z$};
        
    \end{tikzpicture}
    }
    \scalebox{0.8}{
    \begin{tikzpicture}
        
        \node[] at (-1,2.5) {\Large$X[2]=$};

        \draw[black] (0,0)--(2,5) node[]{};
        \draw[black] (2,5)--(4,0) node[]{};
        \draw[black] (0,0)--(4,0) node[]{};

        \draw[red] (0,0)--(0,0) node[draw, shape = circle, fill = red, minimum size = 0.2cm, inner sep=0pt]{};
        \draw[red] (2,5)--(2,5) node[draw, shape = circle, fill = red, minimum size = 0.2cm, inner sep=0pt]{};
        \draw[red] (4,0)--(4,0) node[draw, shape = circle, fill = red, minimum size = 0.2cm, inner sep=0pt]{};

        \node[left] at (1,2.5) {\normalfont$1$};
        \node[right] at (3,2.5) {\normalfont$1$};
        \node[below] at (2,0) {\normalfont$10^{-10}$};
        
        
        \draw[gray,thick,dashed] (4,2)--(5.8,2) node[]{};
        \node[above] at (4.9,2) {\normalfont$1024$};

        \draw[gray,thick,dashed] (8,2.1)--(9.8,2.1) node[]{};
        \node[above] at (8.9,2.1) {\normalfont$512$};

        \draw[gray,thick,dashed] (11.2,2.2)--(12.6,2.2) node[]{};
        \node[above] at (11.9,2.2) {\normalfont$256$};
        
        
        \draw[black] (6,1.25)--(7,3.75) node[]{};
        \draw[black] (7,3.75)--(8,1.25) node[]{};
        \draw[black] (6,1.25)--(8,1.25) node[]{};

        \draw[red] (6,1.25)--(6,1.25) node[draw, shape = circle, fill = red, minimum size = 0.2cm, inner sep=0pt]{};
        \draw[red] (7,3.75)--(7,3.75) node[draw, shape = circle, fill = red, minimum size = 0.2cm, inner sep=0pt]{};
        \draw[red] (8,1.25)--(8,1.25) node[draw, shape = circle, fill = red, minimum size = 0.2cm, inner sep=0pt]{};
        
        \node[left] at (6.5,2.5) {\normalfont$2^{-1}$};
        \node[right] at (7.5,2.5) {\normalfont$2^{-1}$};
        \node[below] at (7,1.25){\normalfont$2^{-1} 10^{-10}$};


        \draw[black] (10,1.875)--(11,1.875) node[]{};

        \draw[red] (10,1.875)--(10,1.875) node[draw, shape = circle, fill = red, minimum size = 0.2cm, inner sep=0pt]{};
        \draw[red] (11,1.875)--(11,1.875) node[draw, shape = circle, fill = red, minimum size = 0.2cm, inner sep=0pt]{};
        
        \node[below] at (10.5,1.875) {\footnotesize$2^{-2}10^{-10}$};


        \draw[black] (13,2.1875)--(13.5,2.1875) node[]{};

        \draw[red] (13,2.1875)--(13,2.1875) node[draw, shape = circle, fill = red, minimum size = 0.2cm, inner sep=0pt]{};
        \draw[red] (13.5,2.1875)--(13.5,2.1875) node[draw, shape = circle, fill = red, minimum size = 0.2cm, inner sep=0pt]{};

        \node[below] at (13.25,2.1875) {\scriptsize$2^{-3}10^{-10}$};

        \node[] at (14.5,2.5) {\normalsize$\cdots$};
        \node[] at (15,2.5) {\normalsize$z$};
        
    \end{tikzpicture}
    }
\caption{
A drawing of $M$ (top) and $X[2]$ (bottom).
In each triangle $M_i$ the point $a_i$ is at the top, the point $b_i$ is on the bottom left, and the point $c_i$ is on the bottom right.
}
\label{fig:dense-M-and-X2}
\end{figure}

Now let 
\[ X[k] \coloneqq \left(\bigsqcup_{i=0}^{k-1} M_i\right)\sqcup\left(\bigsqcup_{i=k}^{\infty} X_i\right)\sqcup\{z\}.\]
So $X[k]$ has the top point $a_i$ from each $M_i$ removed for all $i\ge k$; see Figure~\ref{fig:dense-M-and-X2}(bottom).
Note that $X[k]$ is $G$-invariant.

We claim $d_\gh(X[k], M) < 2^{-k-1}$.
Indeed, consider the correspondence $R \subseteq X[k] \times M$ that contains $(a_i,a_i),(b_i,b_i),(c_i,c_i)$ for all $0\le i\le k-1$, that contains $(b_i,a_i),(c_i,b_i),(c_i,c_i)$ for all $k\le i<\infty$ (recall Figure~\ref{fig:dGH-Xi-Mi}), and that contains $(z,z)$.
We have $\dis(R)=\dis(R_k)<2^{-k}$, giving $d_\gh(X[k], M) < 2^{-k-1}$.

Since $d_\h(X[k],M)=2^{-k}$, by taking $k$ sufficiently large we can obtain $d_\h(X[k],M)<\delta$ for any $\delta>0$.
We will next show that $d_\gh(X[k]/G,M/G) = 2^{-k-1}$.
This will give
\[d_\gh(X[k]/G, M/G) = 2^{-k-1} > d_\gh(X[k], M),\]
completing the proof of the theorem.

We first show that $d_\gh(X[k]/G,M/G)\geq2^{-k-1}$.
Let $R\subseteq (X[k]/G) \times (M/G)$ be an arbitrary correspondence between $X[k]/G$ and $M/G$, and for the sake of contradiction suppose that $\dis(R)<2^{-k}$.
For each $0\leq i<\infty$, we have $M_i/G=\{[a_i],[b_i]\}$ and $X_i/G=\{[b_i]\}$, where the brackets denote the equivalence class of an element under the action by $G$.
A point in $M/G$ is either in some $M_i/G$ for $0\le i<\infty$ or else is equal to $[z]$.

We first prove that if $(x_0,m)\in R$ with $x_0\in X_0/G$, then $m\in M_0/G$.
We may write $m\in M_i/G$ for some $i$.
Suppose for the sake of contradiction that $i>0$.
Note that $[z]\in X[k]/G$ corresponds to some $m_j\in M_j/G$.
If $i=j$, then $d(m,m_j)\leq1$ so $2^{-k}>|d(x_0,[z])-d(m,m_j)|\geq2^{11}-1$, a contradiction.
Hence, we may write
\[2^{-k}>|d(x_0,[z])-d(m,m_j)|=2^{11}-|2^{11-i}-2^{11-j}|.\]
As a result, $j=0$, yielding $2^{-k}>2^{11-i}$.
Now, notice that $[b_2]\in X_2/G$ corresponds to some $m_s\in M_s/G$.
If $i=s$, then $d(m,m_s)\leq1$ and so $2^{-k}>|d(x_0,[b_2])-d(m,m_s)|\geq (2^{11}-2^9)-1$, which is not the case.
Hence, we may write
\[2^{-k}>|d(x_0,[b_2])-d(m,m_s)|=\left|(2^{11}-2^9)-|2^{11-i}-2^{11-s}|\right|.\]
If $s=0$, we have $2^{-k}>2^9-2^{11-i}\geq2^9-2^{-k}$, which is false.
If $s\geq1$, we have
\[2^{-k} > 2^{11}-2^9-|2^{11-s}-2^{11-i}| = 2^9+2^{10}-|2^{11-s}-2^{11-i}|\geq2^9,\]
which is also false.
Therefore, we may conclude that $i=0$.

More generally, we now prove that if $(x_k,m)\in R$ with $x_k\in X_k/G$, then $m\in M_k/G$.
Consider $[b_0]\in X_0/G$ and $[b_k]\in X_k/G$.
We have $([b_0],m_0)\in R$ for some $m_0\in M_0/G$.
Suppose $([b_k],m_i)\in R$ with $m_i\in M_i/G$.
If $i\neq k$, then
\[
2^{-k} > |d([b_0],[b_k])-d(m_0,m_i)| = |(2^{11}-2^{11-k})-(2^{11}-2^{11-i})| = |2^{11-k}-2^{11-i}|,
\]
and we obtain $2^{-k}>2^{10-\min(k,i)}\geq2^{10-k}$, which is certainly not the case.
Hence, $i=k$.

Now, notice that $[a_k]\in M_k/G$ corresponds to some $x_j\in X_j/G$ for some $j$.
By a similar argument as above we obtain that $j=k$.
Hence $x_j=[b_k]$, yielding
\[2^{-k}>|d([b_k],[b_k])-d([a_k],[b_k])|=d([a_k],[b_k])=2^{-k},\]
a contradiction.
Therefore, $\dis(R)\geq2^{-k}$.
Thus, by Theorem~\ref{thm:distortion_GH}, $d_\gh(X[k]/G,M/G)\geq2^{-k-1}$.2

The other inequality $d_\gh(X[k]/G, M/G) \leq 2^{-k-1}$ follows since the correspondence $S \subseteq X[k]/G \times M/G$ containing $([a_i],[a_i]),([b_i],[b_i])$ for all $0\le i\le k-1$, containing $([b_i],[a_i]),([b_i],[b_i])$ for all $k\le i<\infty$, and containing $([z],[z])$ has distortion $\dis(S)=\diam(M_k/G)=2^{-k}$, giving $d_\gh(X[k], M) \le d_\gh(X_k/G, M_k/G) = \frac{1}{2}\diam(M_k/G) = 2^{-k-1}$.

\end{proof}

\subsection{A generalization in the arbitrarily dense setting}
\label{ssec:results-density-generalization}

We now generalize the construction from the prior section.

Let $M_0$ be an arbitrary nonempty metric space with diameter $0<\eta<\infty$, let $G$ be a group acting properly by isometries on $M_0$, and let $X_0$ be a nonempty $G$-invariant subset of $M_0$.
Define
\[
\delta\coloneqq 2^{1-k}\cdot d_\gh(X_0/G,M_0/G)
\quad\quad\text{ and }\quad\quad
\xi\coloneqq \max\left\{8\delta,\tfrac{4}{3}(\delta+\eta)\right\}.
\]

For each nonnegative integer $i$, define the metric space $M_i$ to be a copy of the metric space $M_0$ with all of the distances scaled by $2^{-i}$.
Let $X_i$ denote the subset of $M_i$ corresponding to the scaled copy of $X_0$.
Let $M_{\infty}\coloneqq\{m_{\infty}\}$ be a metric space consisting of a single point, and let $X_{\infty}\coloneqq M_{\infty}$.
Define
\[
M\coloneqq\bigsqcup_{0\leq i\leq\infty}M_i.
\]
Equip $M$ with a metric $d$ whose restriction to each $M_i$ is just the corresponding metric in $M_i$ and for which $d(m_i,m_j)=\xi\cdot|2^{-i}-2^{-j}|$ for all $m_i\in M_i$ and $m_j\in M_j$ with $i\neq j$.
In particular, $d(m_i,m_\infty)=\xi\cdot 2^{-i}$.
For each nonnegative integer $k$, define
\[
X[k]\coloneqq\left(\bigsqcup_{0\leq i\leq k-1}M_i\right)\sqcup\left(\bigsqcup_{k\leq i\leq\infty}X_i\right).
\]
Clearly we have $d_\h(X[k],M)=2^{-k}\cdot d_\h(X_0,M_0)$ and $d_\h(X[k]/G,M/G)=2^{-k}\cdot d_\h(X_0/G,M_0/G)$.
Also, it is not difficult to see that 
$
d_\gh(X[k],M)\leq2^{-k}\cdot d_\gh(X_0,M_0)
$.

\begin{claim}
\label{claim:arb-dense-gen}
We claim that $d_\gh(X[k]/G,M/G)\geq\frac{\delta}{2}=2^{-k}\cdot d_\gh(X_0/G,M_0/G)$.
\end{claim}

\begin{proof}

For the sake of contradiction, suppose there were a correspondence $R\subseteq (X[k]/G) \times (M/G)$ between $X[k]/G$ and $M/G$ such that $\dis(R)<\delta$.
Then $M/G$ is the disjoint union of all $M_i/G$ with $0\leq i\leq\infty$, and $X[k]/G$ is the disjoint union of all $M_0 /G,\ldots, M_{k-1}/G$ and all $X_k/G,X_{k+1}/G,\ldots,X_{\infty}/G$.

We first prove that if $(x_0,m)\in R$ with $x_0\in X_0/G$, then $m\in M_0/G$.
We may write $m\in M_i/G$ for some $i$.
Suppose for the sake of contradiction that $i>0$.
Note that $[m_{\infty}]\in X[k]/G$ corresponds to some $m_j\in M_j/G$.
We have that $i\neq j$ as otherwise
\[
\delta>|d(x_0,m_{\infty})-d(m,m_j)|\geq|\xi-\eta|\geq\tfrac{4}{3}(\delta+\eta)-\eta>\delta.
\]
Since $i\neq j$, we may write
\[
\delta>|d(x_0,m_{\infty})-d(m,m_j)|=\xi\cdot(1-|2^{-i}-2^{-j}|).
\]
If $j>0$, then $\delta>\xi/4$, a contradiction.
Hence, $j=0$, yielding $\delta>\xi\cdot2^{-i}$.
Now, there exists some $x_2\in X_2/G$ that corresponds to some $m_s\in M_s/G$.
We have $i\neq s$ as otherwise
\[
\delta>|d(x_0,x_2)-d(m,m_s)|\geq\left|\tfrac{3}{4}\xi-\eta\right|\geq\delta.
\]
Since $i\neq s$, we may write
\[\delta>|d(x_0,x_2)-d(m,m_s)|=\xi\cdot\left|\tfrac{3}{4}-|2^{-i}-2^{-s}|\right|.\]
Note $s=0$ is impossible, since it would give $\delta>\xi\cdot(\frac{1}{4}-2^{-i})>\frac{\xi}{4}-\delta$ and hence $\delta>\frac{\xi}{8}$, which is false.
Note $s\ge 1$ is also impossible, since it would give $|2^{-i}-2^{-s}|\leq\frac{1}{2}$ and hence $\delta>\frac{\xi}{4}$, which is false.
Therefore, we may conclude that $i=0$.

More generally, we now prove that if $(x_k,m)\in R$ with $x_k\in X_k/G$, then $m\in M_k/G$.
Suppose $m_i \in M_i/G$.
Fix $x_0\in X_0/G$, so $(x_0,m_0)\in R$ for some $m_0\in M_0/G$.
If $i\neq k$, then
\[
\delta > |d(x_0,x_k)-d(m_0,m)| = \xi\cdot|(1-2^{-k})-(1-2^{-i})| = \xi\cdot|2^{-i}-2^{-k}|
\]
and we obtain $\delta>\xi\cdot2^{-1-\min(k,i)}\geq\frac{\xi}{2}$, which is certainly not the case.
Hence, $i=k$.
Through a similar argument, if $x_k\in X_k/G$ satisfies $(x_k,m)\in R$, then $m\in M_k/G$.
As a result, the restriction $R'$ of $R$ to $M_k/G$ is a correspondence between $X_k/G$ and $M_k/G$.
We then have that
\[
\delta=2^{1-k}d_\gh(X_0/G,M_0/G)=2d_\gh(X_k/G,M_k/G)\leq\dis(R')\leq\dis(R)<\delta,
\] a contradiction.
Therefore, any correspondence between $X[k]/G$ and $M/G$ has distortion at least $\delta$, meaning
\[
d_\gh(X[k]/G,M/G)\geq\tfrac{\delta}{2}=2^{-k}\cdot d_\gh(X_0/G,M_0/G).
\]
\end{proof}

\subsection{Arbitrary density with arbitrary constants}
\label{ssec:arb-dense-arb-constants}

We now show that even for arbitrarily dense $G$-invariant subsets $X\subseteq M$, the ratio $\frac{d_\gh(X/G,M/G)}{d_\gh(X,M)}$ can be made either arbitrarily large or arbitrarily small.

Indeed, we show that for any positive integer $n$, there is a compact metric space $M$ with a proper isometric $G$ action such that for any $\delta>0$, there is a $G$-invariant subset $X\subseteq M$ satisfying $d_\h(X,M)<\delta$ and $d_\gh(X/G,M/G) \ge n\cdot d_\gh(X,M)$.
Similarly, for any positive integer $n$, there is a compact metric space $M$ with a proper isometric $G$ action such that for any $\delta>0$, there is a $G$-invariant subset $X\subseteq M$ satisfying $d_\h(X,Z)<\delta$ and $d_\gh(X/G,M/G) \le \frac{1}{n}\cdot d_\gh(X,M)$.

\begin{theorem}
\label{thm:gh-ratio-dense}
For any positive integer $n$, there is a compact metric space $M$ with a proper isometric $G$ action such that for any $\delta>0$, there is a $G$-invariant subset $X\subseteq M$ with $d_\h(X,M)<\delta$ and $d_\gh(X/G,M/G) \le \frac{1}{n}\cdot d_\gh(X,M)$.
\end{theorem}

\begin{proof}[Proof of Theorem~\ref{thm:gh-ratio-dense}
]
Let $X_0\subseteq M_0$ from Theorem~\ref{thm:gh-ratio}(b) be the finite $G$-invariant metric spaces satisfying $d_\gh(X_0/G, M_0/G) \geq n\cdot d_\gh(X_0, M_0)$.
We now input these spaces into the construction of Claim~\ref{claim:arb-dense-gen}.
Since $X_0$ and $M$ are finite, this construction yields compact $G$-invariant metric spaces $X \subseteq M$.
Fix $k$ sufficiently large such that $d_\h(X[k],M)=2^{-k}d_\h(X_0,M_0)<\delta$.
We have
\begin{align*}
d_\gh(X[k]/G,M/G)
&\ge \tfrac{1}{2^k}\cdot d_\gh(X_0/G,M_0/G) &&\text{by Claim~\ref{claim:arb-dense-gen}} \\
&\ge \tfrac{n}{2^k}\cdot d_\gh(X_0,M_0) &&\text{by Theorem~\ref{thm:gh-ratio}(b)} \\
&= n\cdot d_\gh(X[k],M).
\end{align*}
Therefore, the spaces $X[k]\subseteq M$ satisfy the necessary criteria.
\end{proof}


\begin{remark}
When we didn't insist that the subset $X\subseteq M$ was arbitrarily dense in $M$, we were able to produce the bounds $\frac{d_\gh(X/G,M/G)}{d_\gh(X,M)}\ge n$ or $\frac{d_\gh(X/G,M/G)}{d_\gh(X,M)}\le n$ in Theorems~\ref{thm:gh-ratio} using a \emph{finite} subset $M$.
Now, when want to allow $X\subseteq M$ to be arbitrarily dense, in Theorem~\ref{thm:gh-ratio-dense} we use a more complicated space $M$ that is no longer finite, but that is still compact.
\end{remark}

\section{Addressing a question from~\cite{HvsGH}}
\label{sec:addressing-a-question}

In this section, we leave behind group actions in order to answer a question inspired from~\cite{HvsGH}.
Theorem~1(a) of~\cite{HvsGH} shows that if $X$ is a sufficiently dense sample from a compact manifold $M$, then the Gromov--Hausdorff distance $d_\gh(X,M)$ is bounded from below by a reasonable constant times the Hausdorff distance $d_\h(X,M)$.
Theorem~5 of~\cite{HvsGH} constructs an example showing that if $X\subseteq Z$ is arbitrary (i.e., if $Z$ is not a manifold and if $X$ is not sufficiently dense), then the bound may not hold:
\emph{For any $\varepsilon>0$, there exists a compact metric space $Z$ and a subset $X\subseteq Z$ with
\[d_\gh(X,Z) < \varepsilon \cdot d_\h(X,Z).\]}
So, at least one of the manifold assumption or the density assumption (or variants thereof) is needed in general.

In this section, we show that the density assumption alone does not suffice in general.
Indeed, we show how to mimic the construction from Section~\ref{sec:results-density}, replacing each ``tripod'' $M_i$ with a scaled copy of the construction from~\cite[Theorem~5]{HvsGH}, in order to prove the following stronger result.

\begin{theorem}
\label{thm:manifold-assumtion-necessary}
For any $\varepsilon>0$, there exists a compact metric space $Z$ such that for any $\delta>0$, there is a subset $X\subseteq Z$ with $d_\h(X,Z)<\delta$ with
\[d_\gh(X,Z) < \varepsilon \cdot d_\h(X,Z).\]
\end{theorem}

As we have mentioned, this result shows that the manifold assumption (or a variant thereof) in~\cite[Theorem~1(a)]{HvsGH} is needed, even in the presence of arbitrarily dense subsets.

\begin{proof}
We begin by giving the construction that proves~\cite[Theorem~1(a)]{HvsGH}.
Choose the integer $n$ large enough so that $\frac{1}{\sqrt{n}}<\varepsilon$.
Let $\R^n$ be Euclidean space with the standard $\ell_2$-metric.
We denote the standard basis vectors in $\R^n$ by $e_1, \dots, e_n$.
For $j\in\{1,\ldots ,n\}$, let $x_j = \sum_{i=1}^j i\cdot e_i$.
Let $Z_0 = \{x_1, \dots, x_n\} \subseteq \R^n$, and let $X_0 = \{x_1,\dots, x_{n-1}\} \subseteq Z_0$.
We have $d_\h(X_0,Z_0) = \|x_n - x_{n-1}\| = n$.
To bound the Gromov--Hausdorff distance between $X_0$ and $Z_0$, consider the linear map $f \colon \R^n \to \R^n$ defined by $f(e_i) = e_{i+1}$ for $i \in \{1,\dots, n-1\}$ and $f(e_n) = e_1$.
For $j \in \{2,\dots,n\}$ we have $\|x_j - f(x_{j-1})\| = \|\sum_{i=1}^j e_i\| = \sqrt{j}$.
Further, $\|x_1 - f(x_1)\| = \|e_1 - e_2\| = \sqrt{2}$.
We get that $d_\h(f(X_0),Z_0) = \sqrt{n}$, and since $f$ is an isometry of $\R^n$, this proves that ${d_\gh(X_0,Z_0) \le \sqrt{n}}$.
Thus
\[\frac{d_\gh(X_0,Z_0)}{d_\h(X_0,Z_0)}\le\frac{\sqrt{n}}{n}=\frac{1}{\sqrt{n}}<\varepsilon.\]

We are now prepared to prove Theorem~\ref{thm:manifold-assumtion-necessary}.
For $i\ge 0$, define the metric space $Z_i\subseteq \R^n$ to be a scaled copy of $Z_0$ in which all distances are multiplied by $2^{-i}$.
Let
\[ Z=\{z\} \sqcup \bigsqcup_{i=0}^{\infty} Z_i.\]
As in the proof of Theorem~\ref{thm:GH-Hausdorff_dense}, we set the distance between any two points in $Z_i$ and $Z_j$ for $i<j$ to be $\sum_{k=i}^{j-1}2^{10-k} = 2^{11-i} - 2^{11-j}$; note this is \emph{not} their distance in Euclidean space.
Also,we set the distance between $z$ and any point in $Z_i$ to be $\sum_{k=i}^{\infty}2^{10-k} = 2^{11-i}$.
Including the point $z$ ensures that the metric space $Z$ is compact.

Given $\delta>0$ (and $n$ such that again $\frac{1}{\sqrt{n}}<\varepsilon$), choose $k$ such that $2^{-k}\cdot n < \delta$.
For $i\ge 0$, define the submetric space $X_i\subseteq Z_i \subseteq \R^n$ to be a scaled copy of $X_0\subseteq Z_0$ in which all distances are multiplied by $2^{-i}$.
Now let 
\[ X[k] \coloneqq \{z\}\sqcup \bigsqcup_{i=0}^{k-1} Z_i\sqcup\bigsqcup_{i=k}^{\infty} X_i.\]
Note we have the subset relationship $X[k]\subseteq Z$.
Observe that the Hausdorff distance  between $X[k]$ and $Z$ satisfies
$d_\h(X[k],Z) = d_\h(X_k,Z_k) = 2^{-k}\cdot n < \delta$, which is the required density assumption on the subset $X[k] \subseteq Z$.
Furthermore, by considering the identity function $Z_i \to Z_i$ for $0\le i\le k-1$, and by considering a copy of the rotation $f\colon \R^n \to \R^n$ for each $k \le i$, we obtain the following bound on the Gromov--Hausdorff distance between $X[k]$ and $Z$:
\[d_\gh(X[k],Z)\le d_\gh(X_k,Z_k)\le d_\h(f(X_k),Z_k)=2^{-k}\cdot\sqrt{n}.\]
Hence
\[\frac{d_\gh(X[k],Z)}{d_\h(X[k],Z)}\le\frac{2^{-k}\sqrt{n}}{2^{-k}\cdot n}=\frac{1}{\sqrt{n}}<\varepsilon,\]
as required.
\end{proof}

\section{Conclusion and open questions}
\label{sec:conclusion}

We have studied the Hausdorff and Gromov--Hausdorff distances between two metric spaces and quotients thereof by a group action.
In particular, we found that the Hausdorff distances between such spaces have well-behaved relationships.
That is, for a metric space $M$ equipped with a proper isometric $G$ action and for $X\subseteq M$ G-invariant, we have $d_\h(X/G,M/G)=d_\h(X,M)$.
By contrast, the Gromov--Hausdorff distances of the same spaces do not generally preserve these relationships.
Indeed, under the same conditions, there is no general relationship between $d_\gh(X/G,M/G)$ and $d_\gh(X,M)$, even when $M$ compact and $X\subseteq M$ is arbitrarily dense.

Our research process has touched upon several open questions, which we list below:

\begin{enumerate}

\item
If $C$ is a convex region in the plane, with boundary $\partial C$, how does $d_\gh(\partial C,C)$ compare to $d_\h(\partial C,C)$?
This question is interesting already in the case when $C$ is a convex polygon.

\item
More generally, if $C$ is a convex region in $\R^n$, how does $d_\gh(\partial C,C)$ compare to $d_\h(\partial C,C)$?
This question is interesting already in the case of platonic solids.

\item
What is the Gromov--Hausdorff distance $d_\gh(B^{m+1},S^m)$ between the $(m+1)$-dimensional Euclidean ball $B^{m+1}=\{x\in \R^{m+1}~|~\|x\|\le 1\}$ and its $m$-dimensional boundary sphere $S^m=\{x\in \R^{m+1}~|~\|x\|\le 1\}$, when both are equipped with the Euclidean metric?
We refer the reader to~\cite{312046} for some discussion of the case $m=1$, which is already hard.

\item
Let $B^n=\{(x_1,\ldots,x_n)~|~x_1^2+\ldots+x_n^2 \le 1\}$ be the $n$-dimensional unit ball, equipped with the Euclidean metric.
What is the Gromov--Hausdorff distance between $B^n$ and $B^m$ for $n\neq m$?

\item
Let $M_g$ and $M_{g'}$ be two orientable Riemannian surfaces of genus $g$.
If the curvature of $M_g$ and $M_{g'}$ are both bounded by $\kappa$, then can we prove that $d_\gh(M_g,M_{g'})$ is lower bounded by some function of $\kappa$ and $|g-g'|$?

This question is already interesting in the case when $g=0$ and $g'=1$, i.e.\ when $M_0$ is the sphere and when $M_1$ is the flat torus (the product of two circles), both scaled to have the same diameter.
A first lower bound could be obtained from the stability of persistent homology~\cite{chazal2009gromov}, but we expect better bounds to be possible.

A classic example showing that Gromov--Hausdorff convergence of manifolds does not preserve topology is the case of a torus with a smaller and smaller handle converging to a sphere.
However, note that in this convergent sequence, the curvature in the shrinking handle is blowing up without bound.
That is why we have included a curvature bound in the statement of this question.

\end{enumerate}

\bibliographystyle{plain}
\bibliography{GromovHausdorffDistancesBetweenQuotientMetricSpaces}

\begin{thebibliography}{10}

\bibitem{HvsGH}
Henry Adams, Florian Frick, Sushovan~Majhi Majhi, and Nicholas McBride.
\newblock Hausdorff vs {G}romov--{H}ausdorff distances.
\newblock {\em arXiv preprint arXiv:2309.16648}, 2023.

\bibitem{agarwal2018computing}
Pankaj~K Agarwal, Kyle Fox, Abhinandan Nath, Anastasios Sidiropoulos, and Yusu
  Wang.
\newblock Computing the {G}romov--{H}ausdorff distance for metric trees.
\newblock {\em ACM Transactions on Algorithms (TALG)}, 14(2):1--20, 2018.

\bibitem{berger2000encounter}
Marcel Berger.
\newblock Encounter with a geometer, {II}.
\newblock {\em Notices Amer. Math. Soc.}, 47:326--340, 2000.

\bibitem{312046}
Ilya Bogdanov.
\newblock Gromov--{H}ausdorff distance between a disk and a circle.
\newblock MathOverflow.
\newblock \url{https://mathoverflow.net/q/312046}.

\bibitem{bridson2011metric}
Martin~R Bridson and Andr{\'e} Haefliger.
\newblock {\em Metric spaces of non-positive curvature}, volume 319.
\newblock Springer Science \& Business Media, 2011.

\bibitem{BuragoBuragoIvanov}
Dmitri Burago, Yuri Burago, and Sergei Ivanov.
\newblock {\em A course in metric geometry}, volume~33.
\newblock American Mathematical Society, Providence, 2001.

\bibitem{chazal2009gromov}
Fr{\'e}d{\'e}ric Chazal, David Cohen-Steiner, Leonidas~J Guibas, Facundo
  M{\'e}moli, and Steve~Y Oudot.
\newblock Gromov--{H}ausdorff stable signatures for shapes using persistence.
\newblock In {\em Computer Graphics Forum}, volume~28, pages 1393--1403, 2009.

\bibitem{dong2021gromov}
Meihua Dong, Keonhee Lee, and Carlos Morales.
\newblock Gromov--{H}ausdorff stability for group actions.
\newblock {\em Discrete \& Continuous Dynamical Systems: Series A}, 41(3),
  2021.

\bibitem{dummit2004abstract}
David~Steven Dummit and Richard~M Foote.
\newblock {\em Abstract algebra}, volume~3.
\newblock Wiley Hoboken, 2004.

\bibitem{edwards1975structure}
David~A Edwards.
\newblock The structure of superspace.
\newblock In {\em Studies in topology}, pages 121--133. Elsevier, 1975.

\bibitem{graczyk2022model}
Piotr Graczyk, Hideyuki Ishi, Bartosz Ko{\l}odziejek, and H{\'e}l{\`e}ne
  Massam.
\newblock Model selection in the space of {G}aussian models invariant by
  symmetry.
\newblock {\em The Annals of Statistics}, 50(3):1747--1774, 2022.

\bibitem{gromov1981groups}
Michael Gromov.
\newblock Groups of polynomial growth and expanding maps (with an appendix by
  {J}acques {T}its).
\newblock {\em Publications Math{\'e}matiques de l'IH{\'E}S}, 53:53--78, 1981.

\bibitem{gromov1981structures}
Mikhael Gromov.
\newblock Structures m{\'e}triques pour les vari{\'e}t{\'e}s {R}iemanniennes.
\newblock {\em Textes Math{\'e}matiques}, 1, 1981.

\bibitem{hausdorff1914grundzuge}
Felix Hausdorff.
\newblock {\em Grundz{\"u}ge der mengenlehre}, volume~7.
\newblock von Veit, 1914.

\bibitem{huang2024approximately}
Ningyuan Huang, Ron Levie, and Soledad Villar.
\newblock Approximately equivariant graph networks.
\newblock {\em Advances in Neural Information Processing Systems}, 36, 2024.

\bibitem{ivanov2016realizations}
Alexander Ivanov, Stavros Iliadis, and Alexey Tuzhilin.
\newblock Realizations of {G}romov--{H}ausdorff distance.
\newblock {\em arXiv preprint arXiv:1603.08850}, 2016.

\bibitem{kalton1999distances}
Nigel~J Kalton and Mikhail~I Ostrovskii.
\newblock Distances between {B}anach spaces.
\newblock {\em Forum Math.}, 11:1--17, 1999.

\bibitem{khan2018gh}
Abdul~Gaffar Khan, Pramod Das, and Tarun Das.
\newblock {GH}-stability and spectral decomposition for group actions.
\newblock {\em arXiv preprint arXiv:1804.05920}, 2018.

\bibitem{korman2015probably}
Simon Korman, Roee Litman, Shai Avidan, and Alex Bronstein.
\newblock Probably approximately symmetric: Fast rigid symmetry detection with
  global guarantees.
\newblock In {\em Computer Graphics Forum}, volume~34, pages 2--13. Wiley
  Online Library, 2015.

\bibitem{memoli2008euclidean}
Facundo M{\'e}moli.
\newblock Gromov--{H}ausdorff distances in {E}uclidean spaces.
\newblock In {\em 2008 IEEE Computer Society Conference on Computer Vision and
  Pattern Recognition Workshops}, pages 1--8, 2008.

\bibitem{ms04}
Facundo M{\'e}moli and Guillermo Sapiro.
\newblock Comparing point clouds.
\newblock In {\em Proceedings of the 2004 Eurographics/ACM SIGGRAPH Symposium
  on Geometry Processing}, pages 32--40, 2004.

\bibitem{ms05}
Facundo M{\'e}moli and Guillermo Sapiro.
\newblock A theoretical and computational framework for isometry invariant
  recognition of point cloud data.
\newblock {\em Foundations of Computational Mathematics}, 5(3):313--347, 2005.

\bibitem{mgp_approx_symm_sig_06}
N.~J. Mitra, L.~Guibas, and M.~Pauly.
\newblock Partial and approximate symmetry detection for 3d geometry.
\newblock {\em ACM Transactions on Graphics (SIGGRAPH)}, 25(3):560--568, 2006.

\bibitem{schmiedl2015shape}
Felix Schmiedl.
\newblock {\em Shape matching and mesh segmentation}.
\newblock PhD thesis, Technische Universit{\"a}t M{\"u}nchen, 2015.

\bibitem{schmiedl2017computational}
Felix Schmiedl.
\newblock Computational aspects of the {G}romov--{H}ausdorff distance and its
  application in non-rigid shape matching.
\newblock {\em Discrete \& Computational Geometry}, 57(4):854, 2017.

\bibitem{sormani2011intrinsic}
Christina Sormani and Stefan Wenger.
\newblock The intrinsic flat distance between riemannian manifolds and integral
  current spaces.
\newblock {\em Journal of Differential Geometry}, 87(1):117--199, 2011.

\bibitem{tuzhilin2016invented}
Alexey~A Tuzhilin.
\newblock Who invented the {G}romov--{H}ausdorff distance?
\newblock {\em arXiv preprint arXiv:1612.00728}, 2016.

\bibitem{tuzhilin2020lectures}
Alexey~A Tuzhilin.
\newblock Lectures on {H}ausdorff and {G}romov--{H}ausdorff distance geometry.
\newblock {\em arXiv preprint arXiv:2012.00756}, 2020.

\end{thebibliography}

\end{document}